\documentclass{article}
\usepackage{amssymb}

\newtheorem{theorem}{Theorem}

\newtheorem{corollary}[theorem]{Corollary}

\newtheorem{proposition}[theorem]{Proposition}

\newenvironment{proof}[1][Proof]{\noindent\textbf{#1.} }{\ \rule{0.5em}{0.5em}}
\input{tcilatex}

\begin{document}

\title{Ferrers functions of arbitrary degree and order and related functions}
\author{Pinchas Malits \\
(PERI, Physics and Engineering Research Institute\\
at Ruppin Academic Center, Emek Hefer 40250,\\
Israel)}
\maketitle

\section{\protect\bigskip Introduction}

Ferrers functions $P_{\nu }^{-\mu }\left( x\right) $ (associated Legendre
functions of the first kind on the cut) are of utmost importance in
theoretical, applied and computational mathematics. They are solutions of
the Legendre differential equation%
\[
\frac{d}{dx}\left( \left( 1-x^{2}\right) \frac{dy}{dx}\right) +\left( \nu
\left( \nu +1\right) -\frac{\mu ^{2}}{1-x^{2}}\right) y=0 
\]%
on the interval $-1<x<1$ and have found wide applications in analysis,
numerical methods, classical and quantum physics, mechanics and engineering.
In\ particular, $P_{\nu }^{-\mu }\left( x\right) $ arise as a result of the
separation of variables in various physical problems, in expansions of
functions in series and integrals of Ferrers functions as well as in
analytical and numerical approximations based on using orthogonal
polynomials or functions. There are numerous publications in this classical
field, and recent articles by Bakaleinikov and Silbergleit \cite{Bakal},
Bielski \cite{Bie}, Cohl and Costas-Santos \cite{Cohl2}, \cite{Cohl1}, Cohl,
Dang and Dunster \cite{Cohl}, Durand \cite{Dur}, \cite{Dur1}, Maier \cite%
{Maier1}, \cite{Maier2}, Nemes and Olde Daalhuis \cite{Nem}, Szmytkowski 
\cite{Szm}, \cite{Szm1}, Wang and Qiao \cite{Wang}, and Zhou \cite{Zhou}
should be mentioned in this connection. In this article, we obtain for
Ferrers functions novel integral representations, which are used, together
with analytical continuation, for the systematic derivation of numerous
series representations, integral and series connection formulas, asymptotic
and differentiation formulas, generating functions and additional theorems
for $P_{\nu }^{-\mu }\left( \tanh \left( \alpha +\beta \right) \right) $.
Various formulas for Gegenbauer polynomials and\ associated Legendre
functions (associated Legendre polynomials) are obtained as special cases.
Surprisingly, most results have not been given in the literature before. New
derivations, refinements or new forms are suggested for a number of known
relations as well. Because results of this paper are addressed also to
physicists and engineers, the author tried his utmost to attain an
accessible, simple and detailed style of the article and proofs.

In this paper, traditional definitions\ and notations for special functions
are accepted, $\mathbb{R}$ is the set of real numbers, $\mathbb{C}$ is the
set of complex numbers, $\mathbb{C}^{2}=\mathbb{C\times C}$, $\mathbb{Z}$ is
the set of integers, $\mathbb{N}$ is the set of natural numbers and $\mathbb{%
N}_{0}$ is the set of whole numbers.

Ferrers functions of the first kind are defined by means of the Gauss
hypergeometric function \cite{Erdelyi}, \cite{Hobson}, namely

\begin{equation}
P_{\nu }^{-\mu }\left( x\right) =P_{-1-\nu }^{-\mu }\left( x\right) =\left( 
\frac{1-x}{1+x}\right) ^{\frac{\mu }{2}}\frac{F\left( -\nu ,\nu +1;1+\mu ;%
\frac{1-x}{2}\right) }{\Gamma \left( 1+\mu \right) }\text{, }\left\vert
x\right\vert <1\text{,}  \label{1}
\end{equation}%
or on using the Pfaff transformation for Gauss functions 
\begin{equation}
P_{\nu }^{-\mu }\left( x\right) =\frac{2^{-\nu }\left( 1-x\right) ^{\frac{%
\mu }{2}}}{\Gamma \left( 1+\mu \right) \left( 1+x\right) ^{\frac{\mu }{2}%
-\nu }}F\left( -\nu ,\mu -\nu ;1+\mu ;\frac{x-1}{x+1}\right) \text{.}
\label{Hyper1}
\end{equation}

Ferrers functions and their derivatives are entire functions of each of the
parameters $\nu $ and $\mu $. This analyticity allows us to use analytic
continuation when a certain relation is valid in some range of the
parameters values. In many cases, to make use of analytic continuation, we
will exploit the following simple proposition arising from the uniform
convergence of the series.

\begin{proposition}
Suppose that $\mathfrak{G}\subset \mathbb{C}^{2}$ is a bounded domain and
functions $Q_{n}\left( \nu ,\mu \right) $, $n\in $ $\mathbb{N}_{0}$, are
analytic functions of $\nu $ and $\mu $ on $\overline{\mathfrak{G}}$ such
that for certain $a,b,c\in \mathbb{R}$, $\left\vert Q_{n}\left( \nu ,\mu
\right) n^{a\mu +b\nu +c}\right\vert =O\left( 1\right) $ as $n\rightarrow
\infty $. Then $\sum_{n=0}^{\infty }Q_{n}\left( \nu ,\mu \right) x^{n}$ is
an analytic on $\mathfrak{G}$ function of each of the variables $\nu $ and $%
\mu $ if $\left\vert x\right\vert =1$ while $a\func{Re}\mu +b\func{Re}\nu
+c>1$ for all $\left( \nu ,\mu \right) \in \overline{\mathfrak{G}}$ or if $%
\left\vert x\right\vert <1$.
\end{proposition}

One can see from (\ref{Hyper1}) that $P_{k}^{m}\left( x\right) \equiv 0$ as $%
m,k\in \mathbb{N}_{0}$ and $m>k$. If $\mu =\nu -n$ or $\nu =n$, $n\in 
\mathbb{N}_{0}$, then the right side of (\ref{Hyper1}) is proportional to
Jacobi polynomials \cite{Stegun}, \cite{NIST}: 
\begin{eqnarray}
P_{\nu }^{n-\nu }\left( x\right) &=&\frac{n!\left( 1-x^{2}\right) ^{\frac{%
\nu -n}{2}}}{2^{\nu -n}\Gamma \left( \nu +1\right) }P_{n}^{(\nu -n,\nu
-n)}\left( x\right) \text{,}  \label{Jac0} \\
\lim_{\eta \rightarrow k}\Gamma \left( -\eta \right) P_{\eta }^{n+\eta
+1}\left( x\right) &=&2^{k+n+1}n!P_{n}^{(-k-n-1,-k-n-1)}\left( x\right) 
\text{, }  \nonumber
\end{eqnarray}%
and%
\begin{equation}
P_{n}^{-\mu }\left( x\right) =\left( \frac{1-x}{1+x}\right) ^{\frac{\mu }{2}}%
\frac{n!P_{n}^{\left( \mu ,-\mu \right) }\left( x\right) }{\Gamma \left(
n+\mu +1\right) }\text{.}  \label{Jac}
\end{equation}%
Since $P_{n}^{\left( \alpha ,\beta \right) }\left( x\right) =\left(
-1\right) ^{n}P_{n}^{\left( \beta ,\alpha \right) }\left( -x\right) $, the
above relations entails%
\begin{equation}
P_{\nu }^{n-\nu }\left( x\right) =\left( -1\right) ^{n}P_{\nu }^{n-\nu
}\left( -x\right) \text{, \ }P_{k}^{-\mu }\left( x\right) =\left( -1\right)
^{k}\frac{\Gamma \left( k-\mu +1\right) }{\Gamma \left( k+\mu +1\right) }%
P_{k}^{\mu }\left( -x\right) \text{.}  \label{Jac1}
\end{equation}%
Also, (\ref{Jac0}) and (\ref{Jac}) mean that $\{P_{n+\mu }^{-\mu }\left(
x\right) \}$, $\func{Re}\mu >-1$, and $\{P_{n}^{-\mu }\left( x\right) \}$, $%
\left\vert \func{Re}\mu \right\vert <1$, constitute complete systems of
orthogonal functions as $n\in \mathbb{N}_{0}$. In fact, $P_{n+\mu }^{-\mu
}\left( x\right) $ are related to Gegenbauer polynomials $C_{n}^{\left( \tau
\right) }\left( x\right) =\left( -1\right) ^{n}C_{n}^{\left( \tau \right)
}\left( -x\right) $, $\tau \in \mathbb{C}$, defined on $\left( -1,1\right) $
by the relation 
\begin{equation}
C_{n}^{\left( \tau \right) }\left( x\right) =\frac{\left( 2\tau \right) _{n}%
}{\left( \tau +\frac{1}{2}\right) _{n}}P_{n}^{(\tau -\frac{1}{2},\tau -\frac{%
1}{2})}\left( x\right) =\frac{\left( 2\tau \right) _{n}\Gamma \left( \tau +%
\frac{1}{2}\right) P_{n+\tau -\frac{1}{2}}^{\frac{1}{2}-\tau }\left(
x\right) }{2^{\frac{1}{2}-\tau }n!\left( 1-x^{2}\right) ^{\frac{2\tau -1}{4}}%
}\text{.}  \label{Geg0}
\end{equation}%
Gegenbauer polynomials are entire functions of the parameter $\tau $ and
have the generating function \cite{Andr},\cite{Erdelyi}%
\begin{equation}
\left( s^{2}+2sx+1\right) ^{-\tau }=\sum_{n=0}^{\infty }\left( -1\right)
^{n}s^{n}C_{n}^{\left( \tau \right) }\left( x\right) \text{, }\left\vert
s\right\vert <1\text{.}  \label{g-geg}
\end{equation}%
As $n\rightarrow \infty $, there is an asymptotic estimate \cite{Cohl} 
\begin{equation}
\left\vert C_{n}^{\left( \tau \right) }\left( x\right) \right\vert =O\left(
n^{\func{Re}\tau -1}\right) \text{, }-1<x<1\text{.}  \label{As-geg}
\end{equation}

\section{Integral and series representations for Ferrers functions}

Our starting point is employing Euler's integral representation of the Gauss
hypergeometric function (Erdelyi at al. (1955) \cite{Erdelyi}) to write%
\begin{eqnarray}
P_{\nu }^{-\mu }\left( x\right) &=&\frac{2^{-\nu }\left( 1+x\right) ^{\nu -%
\frac{\mu }{2}}\left( 1-x\right) ^{\frac{\mu }{2}}}{\Gamma \left( \mu -\nu
\right) \Gamma \left( 1+\nu \right) }\int_{0}^{1}\frac{s^{\mu -\nu -1}\left(
1-s\right) ^{\nu }ds}{\left( 1-\frac{x-1}{x+1}s\right) ^{-\nu }}\text{,}
\label{Hyper2} \\
\func{Re}\mu &>&\func{Re}\nu >-1\text{, }-1<x<1\text{.}  \nonumber
\end{eqnarray}%
The changes $x=\tanh \alpha $ and $s=\exp \left( -t\right) $ yield the
representation in the form of a Laplace integral, namely%
\begin{eqnarray}
P_{\nu }^{-\mu }\left( \tanh \alpha \right) &=&\frac{2^{\nu }e^{-\alpha \mu
}\cosh ^{-\nu }\alpha }{\Gamma \left( \mu -\nu \right) \Gamma \left( 1+\nu
\right) }\int_{0}^{\infty }\frac{e^{-t\mu }dt}{\left[ \sinh \frac{t}{2}\cosh
\left( \frac{t}{2}+\alpha \right) \right] ^{-\nu }}\text{,}  \label{q1+} \\
\func{Re}\mu &>&\func{Re}\nu >-1\text{.}  \nonumber
\end{eqnarray}

Another integral representation follows in the same way from the relations $%
F\left( a,b,c,z\right) =F\left( b,a,c,z\right) $, (\ref{Hyper1}) and Euler's
integral representations as $\mu =\sigma +\gamma $, $\nu =-\gamma $, $\func{%
Re}\gamma >0$, and $\func{Re}\sigma >-1$, 
\begin{eqnarray}
\frac{P_{-\gamma }^{-\sigma -\gamma }\left( \tanh \alpha \right) }{\cosh
^{\gamma }\alpha } &=&\frac{P_{\gamma -1}^{-\sigma -\gamma }\left( \tanh
\alpha \right) }{\cosh ^{\gamma }\alpha }=\frac{2^{\gamma }e^{-\left( \sigma
+2\gamma \right) \alpha }}{\Gamma \left( \gamma \right) \Gamma \left( \sigma
+1\right) }\int_{0}^{1}\frac{s^{\gamma -1}\left( 1-s\right) ^{\sigma }ds}{%
\left( 1+e^{-2\alpha }s\right) ^{\sigma +2\gamma }}  \nonumber \\
&=&\frac{2^{-\gamma }}{\Gamma \left( \gamma \right) \Gamma \left( \sigma
+1\right) }\int_{0}^{\infty }\frac{\sinh ^{\sigma }\frac{t}{2}dt}{\cosh
^{\sigma +2\gamma }\left( \frac{t}{2}+\alpha \right) }\text{.}  \label{q2}
\end{eqnarray}

By exploiting the binomial series%
\begin{eqnarray*}
\frac{\cosh ^{\gamma }\alpha \sinh ^{\sigma }\frac{t}{2}}{\cosh ^{\sigma
+2\gamma }\left( \frac{t}{2}+\alpha \right) } &=&\frac{\tanh ^{\sigma }\frac{%
t}{2}}{\cosh ^{\sigma +\gamma }\alpha \cosh ^{2\gamma }\frac{t}{2}}\left(
1+\tanh \alpha \tanh \frac{t}{2}\right) ^{-\sigma -2\gamma } \\
&=&\sum_{n=0}^{\infty }\frac{\left( -1\right) ^{n}\left( \sigma +2\gamma
\right) _{n}\tanh ^{n}\alpha \tanh ^{n+\sigma }\frac{t}{2}}{n!\cosh ^{\sigma
+\gamma }\alpha \cosh ^{2\gamma }\frac{t}{2}}
\end{eqnarray*}%
and integrating term-by-term, we obtain from (\ref{q2}), as $\tanh \alpha =x$%
, the representation in terms of a generalized hypergeometric function $%
_{2}R_{1}(a,b;c;\tau ;z)$ \cite{vir}

\begin{eqnarray*}
P_{\gamma -1}^{-\sigma -\gamma }\left( x\right) &=&\frac{2^{-\gamma }\left(
1-x^{2}\right) ^{\frac{\sigma +\gamma }{2}}}{\Gamma \left( \sigma +1\right) }%
\sum_{n=0}^{\infty }\frac{\left( \sigma +2\gamma \right) _{n}\Gamma \left( 
\frac{n+\sigma +1}{2}\right) }{\Gamma \left( \frac{n+\sigma +2\gamma +1}{2}%
\right) }\frac{\left( -x\right) ^{n}}{n!} \\
&=&\frac{2^{-\sigma -\gamma }\sqrt{\pi }\left( 1-x^{2}\right) ^{\frac{\sigma
+\gamma }{2}}}{\Gamma \left( \frac{\sigma +1}{2}+\gamma \right) \Gamma
\left( \frac{\sigma +2}{2}\right) }{}_{2}R_{1}(\sigma +2\gamma ,\frac{\sigma
+1}{2};\frac{\sigma +1}{2}+\gamma ;\frac{1}{2};-x)\text{,}
\end{eqnarray*}%
and hence, by using Legendre's duplication formula for gamma functions, the
representation in terms of a Fox-Wright hypergeometric function ${}_{2}\Psi
_{0}\left( z\right) $ \cite{Fox}, \cite{wri}

\begin{eqnarray}
\frac{P_{\gamma -1}^{-\sigma -\gamma }\left( x\right) }{\left(
1-x^{2}\right) ^{\frac{\sigma +\gamma }{2}}} &=&\frac{2^{\sigma +\gamma -1}}{%
\sqrt{\pi }}\sum_{n=0}^{\infty }\frac{\Gamma \left( \frac{n+\sigma +2\gamma 
}{2}\right) \Gamma \left( \frac{n+\sigma +1}{2}\right) }{\Gamma \left(
\sigma +2\gamma \right) \Gamma \left( \sigma +1\right) }\frac{\left(
-2x\right) ^{n}}{n!}  \label{F0} \\
&=&\frac{2^{\sigma +\gamma -1}}{\sqrt{\pi }\Gamma \left( \sigma +2\gamma
\right) \Gamma \left( \sigma +1\right) }{}_{2}\Psi _{0}\left[ 
\begin{array}{c}
\left( \frac{\sigma }{2}+\gamma ,\frac{1}{2}\right) ,\left( \frac{\sigma }{2}%
+1,\frac{1}{2}\right) \\ 
----%
\end{array}%
;-2x\right] \text{.}  \nonumber
\end{eqnarray}%
By making use of the asymptotic expansion for gamma functions \cite{Stegun}%
\begin{equation}
\frac{\Gamma \left( z+\alpha \right) }{\Gamma \left( z+\beta \right) }%
=z^{\alpha -\beta }+O\left( \frac{1}{z}\right) \text{ as }z\rightarrow
\infty \text{,}  \label{As-gam}
\end{equation}%
one can see that terms of the above power series can be written in the form $%
Q_{n}\left( \sigma ,\gamma \right) n^{a\nu +b\gamma +c}x^{n}$, $\left\vert
Q_{n}\left( \sigma ,\gamma \right) \right\vert \leq Q$, i.e., the power
series converge absolutely and, according to proposition mentioned in
section 1, are analytic functions of $\sigma $ and $\gamma $. This means
that the formulas obtained are valid for all values of $\sigma $,$\gamma \in 
\mathbb{C}$ by virtue of analytic continuation. Note that the even and odd
parts of (\ref{F0}) are proportional to Gauss hypergeometric functions and
hence $P_{\nu }^{-\mu }\left( x\right) $ can be written in the form of the
well-known combination of Gauss functions \cite{Erdelyi}. As $-\sigma
-1=l\in \mathbb{N}_{0}$ or $-\sigma -2\gamma =l\in \mathbb{N}_{0}$, the
series turns into finite sums. On denoting $\gamma =-\lambda $, 
\begin{eqnarray}
P_{\lambda }^{l-\lambda }\left( x\right) &=&\frac{2^{\lambda }l!\left(
1-x^{2}\right) ^{\frac{\lambda -l}{2}}}{\Gamma \left( 2\lambda -l+1\right) }%
\sum_{n=0}^{l}\frac{\Gamma \left( \frac{n-l+2\lambda +1}{2}\right) x^{n}}{%
\Gamma \left( \frac{n-l+1}{2}\right) \left( l-n\right) !n!}  \nonumber \\
&=&\frac{2^{\lambda }l!\left( 1-x^{2}\right) ^{\frac{\lambda -l}{2}}}{\sqrt{%
\pi }\Gamma \left( 2\lambda -l+1\right) }\sum_{j=0}^{\left[ \frac{l}{2}%
\right] }\frac{\left( -1\right) ^{j}\Gamma \left( \lambda -j+\frac{1}{2}%
\right) x^{l-2j}}{2^{2j}j!\left( l-2j\right) !}\text{.}  \label{F}
\end{eqnarray}%
Also, we obtain from (\ref{F0}) the relation%
\begin{equation}
\left( \frac{d^{n}}{dx^{n}}\frac{P_{\nu }^{-\mu }\left( x\right) }{\left(
1-x^{2}\right) ^{\frac{\mu }{2}}}\right) _{x=0}=\frac{\left( -1\right)
^{n}2^{n+\mu -1}\Gamma \left( \frac{n+\mu +\nu +1}{2}\right) \Gamma \left( 
\frac{n+\mu -\nu }{2}\right) }{\sqrt{\pi }\Gamma \left( \mu -\nu \right)
\Gamma \left( \mu +\nu +1\right) }\text{.}  \label{zero}
\end{equation}

New integrals follow from (\ref{q1+}) and (\ref{q2}) by changing $t=4z$ and
using the identity%
\[
\frac{\cosh \left( \alpha +2z\right) }{\cosh ^{2}z\cosh \alpha }=1+2\tanh
\alpha \tanh z+\tanh ^{2}z\text{.} 
\]%
Then, as $x=\tanh \alpha $ and $\tanh z=s$, (\ref{q1+}) results in 
\begin{eqnarray}
P_{\nu }^{-\mu }\left( x\right) &=&\frac{\Gamma ^{-1}\left( 1+\nu \right) }{%
\Gamma \left( \mu -\nu \right) }\left( \frac{1-x}{1+x}\right) ^{\frac{\mu }{2%
}}\int_{0}^{1}\frac{4^{\nu +1}s^{\nu }\left( 1-s\right) ^{2\mu -2\nu -1}ds}{%
\left( 1+s\right) ^{2\mu +2\nu +1}\left( s^{2}+2sx+1\right) ^{-\nu }}\text{,}
\label{new1} \\
\func{Re}\mu &>&\func{Re}\nu >-1\text{, }-1<x<1\text{,}  \nonumber
\end{eqnarray}%
and on denoting $\sigma =\tau +2\nu $ and $\gamma =-\nu $ (\ref{q2}) turns
into%
\begin{eqnarray}
P_{\nu }^{-\nu -\tau }\left( x\right) &=&\frac{2^{\tau +3\nu +2}\left(
1-x^{2}\right) ^{\frac{\tau +\nu }{2}}}{\Gamma \left( -\nu \right) \Gamma
\left( \tau +2\nu +1\right) }\int_{0}^{1}\frac{\left( 1-s^{2}\right) ^{-2\nu
-1}s^{\tau +2\nu }ds}{\left( s^{2}+2sx+1\right) ^{\tau }}\text{,}
\label{new2} \\
\text{ }\func{Re}\nu &<&0\text{, }\func{Re}\left( \tau +2\nu \right) >-1%
\text{, }-1<x<1\text{.}  \nonumber
\end{eqnarray}

\begin{theorem}
Let $\nu ,\mu ,\tau \in \mathbb{C}$. There are the expansions of Ferrers
functions into series of Gegenbauer polynomials: 
\begin{eqnarray}
\frac{P_{\nu }^{-\tau -\nu }\left( x\right) }{\left( 1-x^{2}\right) ^{\frac{%
\tau +\nu }{2}}} &=&\frac{2^{\tau +\nu }\Gamma \left( \frac{1}{2}-\nu
\right) }{\sqrt{\pi }\Gamma \left( \tau +2\nu +1\right) }\sum_{n=0}^{\infty }%
\frac{\left( -1\right) ^{n}\Gamma \left( \frac{n+\tau +2\nu }{2}\right) }{%
\Gamma \left( \frac{n+\tau -2\nu }{2}\right) }C_{n}^{\left( \tau \right)
}\left( x\right) \text{,}  \label{P2} \\
\func{Re}\left( \tau +2\nu \right) &<&0\text{, \ }\nu -\frac{1}{2}\notin 
\mathbb{N}_{0}\text{,}  \label{P22}
\end{eqnarray}%
and%
\begin{eqnarray}
P_{\nu }^{-\mu }\left( x\right) &=&\left( \frac{1-x}{1+x}\right) ^{\frac{\mu 
}{2}}\frac{\Gamma \left( \mu -\nu +\frac{1}{2}\right) }{4^{\nu }\sqrt{\pi }}%
\sum_{n=0}^{\infty }\left( -1\right) ^{n}c_{n}C_{n}^{\left( -\nu \right)
}\left( x\right) \text{,}  \label{P1} \\
\nu -\mu -\frac{1}{2} &\notin &\mathbb{N}_{0}\text{, }\func{Re}\left( 2\mu
-\nu \right) >0\text{,}  \label{P11}
\end{eqnarray}%
where 
\begin{eqnarray*}
c_{n} &=&\frac{2^{2\mu +2\nu +1}\left( \nu +1\right) _{n}}{\Gamma \left(
n+2\mu -\nu +1\right) }F\left( 2\mu +2\nu +1,n+\nu +1;n+2\mu -\nu
+1;-1\right) \\
&=&\frac{\left( \nu +1\right) _{n}}{\Gamma \left( n+2\mu -\nu +1\right) }%
F\left( 2\mu +2\nu +1,2\mu -2\nu ;2\mu +n-\nu +1;\frac{1}{2}\right) \text{.}
\end{eqnarray*}
\end{theorem}

\begin{proof}
First we will examine (\ref{new1}) as $-1<\func{Re}\nu <0$ and $\func{Re}%
\left( 2\mu -2\nu \right) \geq 1$. In this case, as $\tau =-\nu $, the
series in (\ref{g-geg}) converges uniformly on the interval $s\in $ $\left[
0,1\right] $ by virtue of the asymptotic estimate (\ref{As-geg}). Inserting
the above-mentioned series, one is allowed to integrate term-by term. On
evaluating arising integrals with Euler's integral representation of the
Gauss hypergeometric function we arrive at (\ref{P1}) in this initial case.
Now, asymptotic expansions (\ref{As-geg}) and (\ref{As-gam}) manifest that
series in (\ref{P1}) is absolutely convergent under conditions (\ref{P11}),
i.e., it is an analytic function of $\mu $ and $\nu $ (see section 1), and
therefore (\ref{P1}) is valid due to analytic continuation. The expansion (%
\ref{P2}) can be proved in the same manner by examination (\ref{new2}) where
initially is taken $\func{Re}\nu <0$, $\func{Re}\tau <0$, and $\func{Re}%
\left( \tau +2\nu \right) >-1$.
\end{proof}

Since $C_{n}^{\left( \tau \right) }\left( x\right) $ are orthogonal for $%
\func{Re}\tau >-1/2$, coefficients of the series might be expressed\ as
certain new integrals of products of Gegenbauer polynomials and Ferrers
functions while Parseval's equation enables us to evaluate new integrals of
products of Ferrers functions.

\section{\protect\bigskip The generalized Mehler-Dirichlet integral}

Let $I_{\mu }\left( z\right) $ be a modified Bessel function and $\chi
_{\varkappa ,\sigma }\left( \xi ,u\right) $ be a combination of Bessel
functions of the first and second kind, namely 
\[
\chi _{\varkappa ,\sigma }\left( \xi ,u\right) =J_{\varkappa }\left( \xi
e^{-iu/2}\right) Y_{\sigma }\left( \xi e^{iu/2}\right) -Y_{\varkappa }\left(
\xi e^{-iu/2}\right) J_{\sigma }\left( \xi e^{iu/2}\right) \text{.} 
\]

\begin{theorem}
The integral representation%
\begin{eqnarray}
\frac{P_{\nu -\frac{1}{2}}^{\mu }\left( \cos \theta \right) }{\sqrt{\pi }%
\sin ^{\mu }\theta } &=&\frac{\xi ^{\mu +\frac{3}{2}}}{2^{\frac{2\mu +7}{4}}}%
\int_{0}^{\theta }\widehat{\chi }_{\nu }\left( \xi ,u\right) \frac{J_{-\mu -%
\frac{1}{2}}\left( \xi \sqrt{2\left( \cos u-\cos \theta \right) }\right) }{%
\left( \cos u-\cos \theta \right) ^{\frac{1+2\mu }{4}}}du  \label{general} \\
&=&\frac{\xi ^{\mu +\frac{3}{2}}}{2^{\frac{2\mu +7}{4}}}\int_{-\theta
}^{\theta }\chi _{\nu +1,\nu }\left( \xi ,u\right) \frac{J_{-\mu -\frac{1}{2}%
}\left( \xi \sqrt{2\left( \cos u-\cos \theta \right) }\right) }{%
e^{iu/2}\left( \cos u-\cos \theta \right) ^{\frac{1+2\mu }{4}}}du\text{,}
\label{general-a}
\end{eqnarray}%
\[
\widehat{\chi }_{\nu }\left( \xi ,u\right) =e^{-iu/2}\chi _{\nu +1,\nu
}\left( \xi ,u\right) +e^{iu/2}\chi _{\nu +1,\nu }\left( \xi ,-u\right) 
\text{,} 
\]%
is valid on $\theta \in \left( 0,\pi \right) $ as $\xi ,\nu ,\mu \in \mathbb{%
C}$ and $\func{Re}\mu <\frac{1}{2}$.
\end{theorem}

\begin{proof}
We are basing on the integral 
\begin{eqnarray*}
\chi \left( \xi ,\theta \right) &=&\int_{0}^{\theta }\frac{P_{\nu -\frac{1}{2%
}}^{\mu }\left( \cos u\right) I_{\mu -\frac{1}{2}}\left( \xi \sqrt{2\left(
\cos u-\cos \theta \right) }\right) du}{\left( \sqrt{2\left( \cos u-\cos
\theta \right) }\right) ^{\frac{1}{2}-\mu }\sin ^{\mu -1}s}\text{,} \\
\chi \left( \xi ,\theta \right) &=&e^{-\frac{i\pi }{2}}\sqrt{\frac{\pi }{2}}%
\xi ^{\mu -\frac{1}{2}}\chi _{\nu ,\nu }\left( \xi ,\theta \right) \text{,}
\\
\text{\textbf{\ }}\mathbf{-}\frac{1}{2} &<&\func{Re}\left( \mu \right) <1%
\text{, }\theta \in \lbrack 0,\pi )\text{, }\xi \in \mathbb{C}\text{, }\nu
\in \mathbb{C}\text{,}
\end{eqnarray*}%
which follows from the integral (16) of the article \cite{Malits}. The
changes $2x=1-\cos \theta $ and $2t=1-\cos s$ convert the above integral
into a convolution, and then 
\begin{equation}
\int_{0}^{x}f\left( t\right) \frac{I_{\mu -\frac{1}{2}}\left( 2\xi \sqrt{x-t}%
\right) }{\left( x-t\right) ^{\frac{1-2\mu }{4}}}dt=g\left( \xi ,x\right) 
\text{, }0\leq x<1\text{,}  \label{conv}
\end{equation}%
\[
f\left( t\right) =\frac{P_{\nu -\frac{1}{2}}^{\mu }\left( 1-2t\right) }{t^{%
\frac{\mu }{2}}\left( 1-t\right) ^{\frac{\mu }{2}}}\text{, }g\left( \xi
,x\right) =2^{\frac{\mu -1}{2}}\chi \left( \xi ,\arccos \left( 1-2x\right)
\right) \text{.} 
\]%
By taking $f\left( t\right) =0$ as $t\geq 1$ and defining $g\left( \xi
,x\right) $ for $x\geq 1$ as the left side of (\ref{conv}), we redefine (\ref%
{conv}) as an equation on the interval $0\leq x<\infty $. Now, by applying
Laplace integral transform $F\left( s\right) =\mathcal{L}\left( f\left(
x\right) \right) $, one can readily find%
\begin{equation}
F\left( s\right) =\xi ^{\frac{1-2\mu }{2}}s^{\mu +\frac{1}{2}}e^{-\frac{\xi
^{2}}{s}}G\left( s\right) \text{.}  \label{Linv}
\end{equation}%
The inverse Laplace transform of (\ref{Linv}) can be written as a
convolution 
\[
f\left( x\right) \xi ^{\mu -\frac{1}{2}}=\mathcal{L}^{-1}\left( s^{\mu +%
\frac{1}{2}}e^{-\frac{\xi ^{2}}{s}}G\left( s\right) \right) =\mathcal{L}%
^{-1}\left( sG\left( s\right) \right) \ast \mathcal{L}^{-1}\left( s^{\mu -%
\frac{1}{2}}e^{-\frac{\xi ^{2}}{s}}\right) \text{.} 
\]%
By taking into account that $g\left( \xi ,0\right) =0$, the above relation
yields as $0<x<1$,%
\[
\frac{P_{\nu -\frac{1}{2}}^{\mu }\left( 1-2x\right) }{x^{\frac{\mu }{2}%
}\left( 1-x\right) ^{\frac{\mu }{2}}}=\xi \int_{0}^{x}\frac{J_{-\mu -\frac{1%
}{2}}\left( 2\xi \sqrt{x-t}\right) }{\left( x-t\right) ^{\frac{1+2\mu }{4}}}%
d\left( g\left( \xi ,t\right) \right) \text{,} 
\]%
or on going back to the initial variables%
\begin{equation}
\frac{P_{\nu -\frac{1}{2}}^{\mu }\left( \cos \theta \right) }{\sqrt{\pi }e^{-%
\frac{i\pi }{2}}\sin ^{\mu }\theta }=\int_{0}^{\theta }\frac{J_{-\mu -\frac{1%
}{2}}\left( \xi \sqrt{2\left( \cos u-\cos \theta \right) }\right) d\left(
\chi _{\nu ,\nu }\left( \xi ,u\right) \right) }{2^{\frac{2\mu +3}{4}}\xi
^{-\mu -\frac{1}{2}}\left( \cos u-\cos \theta \right) ^{\frac{1+2\mu }{4}}}%
\text{,}  \label{general1}
\end{equation}%
where by virtue of\ analytic continuation $\func{Re}\mu <\frac{1}{2}$.
Finally, by exploiting formulas for derivatives of Bessel functions, we
obtain (\ref{general}), and then (\ref{general-a}).
\end{proof}

Make the changes $\theta =\pi -\vartheta $ and $u=\pi -s$ in (\ref{general}%
). Then, by exploiting (\ref{Jac1}) and analytic continuations of
cylindrical functions \cite{Stegun}, one can obtain as $\nu =n-1/2$, $n\in 
\mathbb{N}$, and $\xi =i\eta $:

\begin{corollary}
As $n\in \mathbb{N}$, $\mu ,\eta \in \mathbb{C}$, $\func{Re}\mu <\frac{1}{2}$%
, and $0<\vartheta <\pi $, 
\begin{equation}
\frac{P_{n-1}^{-\mu }\left( \cos \vartheta \right) }{\sqrt{\pi }\sin ^{\mu
}\vartheta }=\frac{\eta ^{\mu +\frac{3}{2}}\Gamma \left( n-\mu \right) }{2^{%
\frac{2\mu +3}{4}}\Gamma \left( n+\mu \right) }\int_{\vartheta }^{\pi
}W_{n}\left( \eta ,s\right) \frac{I_{-\mu -\frac{1}{2}}\left( \eta \sqrt{%
2\left( \cos \vartheta -\cos s\right) }\right) }{\left( \cos \vartheta -\cos
s\right) ^{\frac{1+2\mu }{4}}}ds\text{,}  \label{general-b}
\end{equation}%
\begin{eqnarray*}
W_{n}\left( \eta ,s\right) &=&\frac{e^{-is/2}\widetilde{\chi }_{n}\left(
\eta ,s\right) -e^{is/2}\widetilde{\chi }_{n}\left( \eta ,-s\right) }{2i}%
\text{,} \\
\widetilde{\chi }_{n}\left( \eta ,u\right) &=&J_{n+\frac{1}{2}}\left( \eta
e^{is/2}\right) Y_{n-\frac{1}{2}}\left( \eta e^{-is/2}\right) +Y_{n+\frac{1}{%
2}}\left( \eta e^{is/2}\right) J_{n-\frac{1}{2}}\left( \eta e^{-is/2}\right) 
\text{.}
\end{eqnarray*}
\end{corollary}

As $\xi =0$ and $\nu =\gamma +1/2$, the integral representation (\ref%
{general}) turns into\ the famous Mehler-Dirichlet integral:

\begin{corollary}
As $\gamma ,\mu \in \mathbb{C}$, $\func{Re}\mu <\frac{1}{2}$, and $0<\theta
<\pi $,%
\begin{equation}
P_{\gamma }^{\mu }\left( \cos \theta \right) =\sqrt{\frac{2}{\pi }}\frac{%
\sin ^{\mu }\theta }{\Gamma \left( \frac{1}{2}-\mu \right) }\int_{0}^{\theta
}\frac{\cos \left( \gamma +\frac{1}{2}\right) u}{\left( \cos u-\cos \theta
\right) ^{\frac{1}{2}+\mu }}du\text{.}  \label{Meh}
\end{equation}
\end{corollary}

The above relation can be viewed as a fractional type integral operator
transforming trigonometric functions into Ferrers functions of the first
kind. Another similar\ operator can be obtained from (\ref{general-b}) as $%
\eta =0$,%
\begin{eqnarray}
P_{n-1}^{-\mu }\left( \cos \vartheta \right) &=&\sqrt{\frac{2}{\pi }}\frac{%
\Gamma \left( n-\mu \right) \sin ^{\mu }\vartheta }{\Gamma \left( \frac{1}{2}%
-\mu \right) \Gamma \left( n+\mu \right) }\int_{\vartheta }^{\pi }\frac{\sin
\left( n-\frac{1}{2}\right) s}{\left( \cos \vartheta -\cos s\right) ^{\frac{1%
}{2}+\mu }}ds\text{,}  \label{Meh1} \\
\func{Re}\mu &<&\frac{1}{2}\text{, }n\in \mathbb{N}\text{, }0<\vartheta <\pi 
\text{.}  \nonumber
\end{eqnarray}%
Certain other fractional type integral representations will be derived in
the next section.

\section{Representations Ferrers functions in the form of fractional type
integrals}

As a results of the changes $t=2u-2\alpha $ and $t=u-\alpha $, (\ref{q1+})
becomes 
\begin{eqnarray}
P_{\nu }^{-\mu }\left( \tanh \alpha \right) &=&\frac{2^{1+\nu }e^{\alpha \mu
}\cosh ^{-\nu }\alpha }{\Gamma \left( \mu -\nu \right) \Gamma \left( 1+\nu
\right) }\int_{\alpha }^{\infty }\frac{e^{-2u\mu }du}{\left[ \sinh \left(
u-\alpha \right) \cosh u\right] ^{-\nu }}  \label{shift} \\
&=&\frac{2^{1+\nu }e^{\alpha \mu }}{\Gamma \left( \mu -\nu \right) \Gamma
\left( 1+\nu \right) }\int_{\alpha }^{\infty }\frac{e^{-2u\mu }\cosh ^{2\nu
}udu}{\left( \tanh u-\tanh \alpha \right) ^{-\nu }}  \label{Ab1}
\end{eqnarray}%
and%
\begin{equation}
P_{\nu }^{-\mu }\left( \tanh \alpha \right) =\frac{\cosh ^{-\nu }\alpha }{%
\Gamma \left( \mu -\nu \right) \Gamma \left( 1+\nu \right) }\int_{\alpha
}^{\infty }\frac{e^{-u\mu }du}{\left( \sinh u-\sinh \alpha \right) ^{-\nu }}%
\text{.}  \label{Ab2}
\end{equation}%
\newline
The above formulas hold as $\func{Re}\mu >\func{Re}\nu >-1$.

Now, additional integral representations of Ferrers functions are obtained
by changing variables and parameters: as $\func{Re}\mu <-\func{Re}\nu <1$,

\begin{equation}
P_{\nu }^{\mu }\left( -\tanh \alpha \right) =\frac{2^{\nu +1}e^{\alpha \mu }%
}{\Gamma \left( -\mu -\nu \right) \Gamma \left( 1+\nu \right) }\int_{-\infty
}^{\alpha }\frac{e^{-2u\mu }\cosh ^{2\nu }udu}{\left( \tanh \alpha -\tanh
u\right) ^{-\nu }}\text{,}  \label{Ab3}
\end{equation}%
\begin{equation}
P_{\nu }^{\mu }\left( -\tanh \alpha \right) =\frac{\cosh ^{-\nu }\alpha }{%
\Gamma \left( -\mu -\nu \right) \Gamma \left( 1+\nu \right) }\int_{-\infty
}^{\alpha }\frac{e^{-u\mu }du}{\left( \sinh \alpha -\sinh u\right) ^{-\nu }}%
\text{.}  \label{Ab4}
\end{equation}

The integral representation (\ref{Ab2}) is a source of an asymptotic
expansion of $P_{\lambda }^{-\mu }\left( x\right) $ as $\mu \rightarrow
\infty $ in the right half-plane. We write%
\[
\left( \frac{\sinh u-\sinh \alpha }{\cosh \alpha }\right) ^{\nu }=\left(
u-\alpha \right) ^{\nu }T\left( u\right) \text{,} 
\]%
where $T\left( u\right) $ is an infinitely differentiable function and can
be represented in the vicinity of the point $u=\alpha $ by the Taylor
series, that is,%
\[
\left( \frac{\sinh u-\sinh \alpha }{\cosh \alpha }\right) ^{\nu
}=\sum_{k=0}^{\infty }\frac{t_{k,\nu }\left( \alpha \right) }{k!}\left(
u-\alpha \right) ^{k+\nu }\text{,} 
\]%
with bounded for $\alpha \in \mathbb{R}$ functions $t_{k,\nu }\left( \alpha
\right) $ defined by the relation 
\begin{eqnarray}
t_{k,\nu }\left( \alpha \right) &=&\frac{1}{\cosh ^{\nu }\alpha }\left( 
\frac{d^{k}}{dx^{k}}\left( \frac{\sinh \left( x+\alpha \right) -\sinh \alpha 
}{x}\right) ^{\nu }\right) _{x=0}\text{,}  \nonumber \\
t_{0,\nu }\left( \alpha \right) &=&1\text{, }t_{1,\nu }\left( \alpha \right)
=\frac{\nu }{2}\tanh \alpha \text{, }t_{2,\nu }\left( \alpha \right) =\frac{%
\nu \left( \nu -1\right) }{4}\tanh ^{2}\alpha +\frac{\nu }{3}\text{.} 
\nonumber
\end{eqnarray}%
Now, Watson's lemma gives

\begin{theorem}
Let the lying in the half-plane $\func{Re}\mu \geq \varepsilon >0$ curve $%
\mathcal{L}$ be going to infinity and, as $\mu \rightarrow \infty $ along $%
\mathcal{L}$, $\nu =\nu \left( \mu \right) $ be bounded and $-1<\func{Re}\nu
<\func{Re}\mu $. If $K$ is an arbitrary fixed positive integer, then as $\mu
\rightarrow \infty $ along $\mathcal{L}$,$\ $%
\begin{equation}
e^{\alpha \mu }P_{\nu }^{-\mu }\left( \tanh \alpha \right) =\frac{1}{\Gamma
\left( \mu -\nu \right) }\left( \sum_{k=0}^{K-1}\frac{\left( \nu +1\right)
_{k}t_{k,\nu }\left( \alpha \right) }{\mu ^{k+\nu +1}k!}+O\left( \frac{1}{%
\mu ^{K+\nu +1}}\right) \right) \text{,}  \label{AsPp}
\end{equation}%
uniformly with respect to $\alpha \in \mathbb{R}$. \ 
\end{theorem}

As $\mu \rightarrow \infty ,$ another form of the asymptotic expansion is
given by the terms of the hypergeometric series (\ref{1}). Various other
asymptotic expansions for $P_{\nu }^{-\mu }\left( \tanh \alpha \right) $ of
large order are presented in \cite{Dunster}, \cite{NIST}, \cite{Nem}, \cite%
{Gil}.

The integrals in (\ref{Ab1}), (\ref{Ab2}), (\ref{Ab3}) and (\ref{Ab4}) can
be viewed as fractional type integral operators transforming exponential
functions into Ferrers functions, i.e., operators connecting index integral
transforms whose kernels contain Ferrers functions (see \cite{Belich}, \cite%
{Leb}, \cite{Mandal}) with Fourier and Laplace transforms or as operators
connecting series of exponents (Fourier series of periodic functions in
particular) with series of Ferrers functions.

Fractional type integral representations connecting Ferrers functions with
power functions can be obtained by changing $t=2u-2\alpha $ in (\ref{q2}):%
\begin{eqnarray}
P_{\gamma -1}^{-\sigma -\gamma }\left( \tanh \alpha \right) &=&\frac{%
2^{1-\gamma }\cosh ^{\gamma }\alpha }{\Gamma \left( \gamma \right) \Gamma
\left( \sigma +1\right) }\int_{\alpha }^{\infty }\frac{\sinh ^{\sigma
}\left( u-\alpha \right) }{\cosh ^{2\gamma +\sigma }u}du  \label{shift2} \\
&=&\frac{2^{1-\gamma }\cosh ^{\sigma +\gamma }\alpha }{\Gamma \left( \gamma
\right) \Gamma \left( \sigma +1\right) }\int_{\alpha }^{\infty }\frac{\left(
\tanh u-\tanh \alpha \right) ^{\sigma }}{\cosh ^{2\gamma }u}du\text{,}
\label{Ab7} \\
\func{Re}\gamma &>&0\text{, }\func{Re}\sigma >-1\text{,}  \nonumber
\end{eqnarray}%
and hence, as $\func{Re}\gamma >0$ and $\func{Re}\sigma >-1$,

\begin{equation}
P_{\gamma -1}^{-\sigma -\gamma }\left( -\tanh \alpha \right) =\frac{%
2^{1-\gamma }\cosh ^{\sigma +\gamma }\alpha }{\Gamma \left( \gamma \right)
\Gamma \left( \sigma +1\right) }\int_{-\infty }^{\alpha }\frac{\cosh
^{-2\gamma }udu}{\left( \tanh \alpha -\tanh u\right) ^{-\sigma }}\text{.}
\label{Ab8}
\end{equation}

The integral representations obtained in this section lead to fractional
type integral operators relating Ferrers functions.

\begin{theorem}
There are fractional type integral operators relating Ferrers functions of
different degrees%
\begin{eqnarray}
\frac{P_{\lambda +\rho +1}^{-\mu }\left( \tanh \alpha \right) }{\cosh
^{-\lambda -\rho -1}\alpha } &=&\frac{\Gamma \left( \mu -\rho \right) \Gamma
^{-1}\left( 1+\lambda \right) }{\Gamma \left( \mu -\lambda -\rho -1\right) }%
\int_{\alpha }^{\infty }\frac{P_{\rho }^{-\mu }\left( \tanh s\right) \cosh
^{\rho +1}s}{\left( \sinh s-\sinh \alpha \right) ^{-\lambda }}ds\text{,}
\label{degr1} \\
\func{Re}\left( \mu -\rho \right) &>&\func{Re}\lambda +1>0\text{, } 
\nonumber
\end{eqnarray}%
or Ferrers functions of different orders 
\begin{eqnarray}
\frac{P_{\nu }^{-\mu -\lambda -1}\left( \tanh \alpha \right) }{\cosh ^{\mu
+\lambda +1}\alpha } &=&\frac{1}{\Gamma \left( \lambda +1\right) }%
\int_{\alpha }^{\infty }\frac{\cosh ^{-\mu -2}sP_{\nu }^{-\mu }\left( \tanh
s\right) }{\left( \tanh s-\tanh \alpha \right) ^{-\lambda }}ds\text{,}
\label{Con} \\
\func{Re}\mu &>&-1\text{, }\func{Re}\lambda >-1\text{. }  \nonumber
\end{eqnarray}
\end{theorem}

\begin{proof}
Equations (\ref{degr1}) is obtained by exploiting the integral \ 
\begin{eqnarray}
\left( \eta \left( u\right) -\eta \left( \alpha \right) \right) ^{\lambda
+\rho +1} &=&\frac{1}{B\left( 1+\lambda ,1+\rho \right) }\int_{\alpha }^{u}%
\frac{\left( \eta \left( s\right) -\eta \left( \alpha \right) \right)
^{\lambda }}{\left( \eta \left( u\right) -\eta \left( s\right) \right)
^{-\rho }}d\eta \left( s\right) \text{,}  \label{int1} \\
\func{Re}\lambda &>&-1\text{, }\func{Re}\rho >-1\text{, }u\geq \alpha \text{,%
}  \nonumber
\end{eqnarray}%
where $\eta \left( s\right) $ is a continuous monotonically increasing
function on the interval $s\in $ $\left[ u,\alpha \right] $. Insert (\ref%
{int1}) with $\eta \left( s\right) =\sinh s$ and $\rho =\nu -\lambda -1$
into (\ref{Ab2}). Interchanging the order of integration and computing
arising integrals by exploiting (\ref{Ab2}) yield the proof of (\ref{degr1}%
). Equation (\ref{Con}) is derived in the same way from (\ref{Ab7}) as $\eta
\left( s\right) =\tanh s$, $\rho +\gamma =\mu $, and $\gamma -1=\nu $.
\end{proof}

On making the changes $\sigma =\lambda +1$, $x=\tanh \alpha $, and $t=\tanh
s $, (\ref{Con}) turns into the integral derived by Collins \cite{Collins},
although his proof is somewhat more complicated. Collins noted that the
Mehler-Dirichlet integral (\ref{Meh}) can be obtained from the
above-mentioned relation as $\mu =-1/2$. The connection relation (\ref{degr1}%
) is new. A curious shift operator for exponents is derived from (\ref{degr1}%
) by setting $\mu =\pm 1/2$, $\rho =-\nu -1$, $\lambda =\eta -1/2$, $\tanh
\alpha =\cos \varphi $, and $\tanh s=\cos \theta $ and using expressions of
Ferrers functions of order $\pm 1/2$ in terms of trigonometric functions 
\cite{Stegun}.

\begin{corollary}
\begin{eqnarray*}
e^{i\left( \nu -\eta \right) \varphi } &=&\frac{\Gamma \left( \nu +\frac{1}{2%
}\right) \sin ^{\frac{1}{2}-\nu }\varphi }{\Gamma \left( \eta +\frac{1}{2}%
\right) \Gamma \left( \nu -\eta \right) }\int_{0}^{\varphi }\frac{e^{i\left(
\nu +\frac{1}{2}\right) \theta }\sin ^{\nu -\eta -1}\theta }{\sin ^{\frac{1}{%
2}-\eta \ }\left( \varphi -\theta \right) }d\theta \text{,} \\
0 &<&\varphi ,\theta <\pi \text{, }\func{Re}\nu >\func{Re}\eta >-\frac{1}{2}%
\text{.}
\end{eqnarray*}
\end{corollary}

A novel fractional type integral representation of Ferrers functions, can be
derived from (\ref{degr1}) on setting $\rho =-k-\mu -1$, $k\in \mathbb{N}%
_{0} $, and making use of (\ref{Geg0}):

\begin{corollary}
Let $2\func{Re}\mu +k>\func{Re}\lambda >-1$. Then%
\begin{equation}
P_{\lambda -k-\mu }^{-\mu }\left( \tanh \alpha \right) =\frac{2^{\mu
}k!\Gamma \left( \mu +\frac{1}{2}\right) \cosh ^{k+\mu -\lambda }\alpha }{%
\sqrt{\pi }\Gamma \left( k+2\mu -\lambda \right) \Gamma \left( 1+\lambda
\right) }\int_{\alpha }^{\infty }\frac{C_{k}^{\left( \mu +\frac{1}{2}\right)
}\left( \tanh s\right) ds}{\left( \sinh s-\sinh \alpha \right) ^{-\lambda
}\cosh ^{k+2\mu }s}\text{.}  \label{cc0}
\end{equation}%
In particular, as $2\func{Re}\mu >\func{Re}\lambda >-1$,%
\begin{equation}
P_{\lambda -\mu }^{-\mu }\left( \tanh \alpha \right) =\frac{2^{\mu }\Gamma
\left( \mu +\frac{1}{2}\right) \cosh ^{\mu -\lambda }\alpha }{\sqrt{\pi }%
\Gamma \left( 2\mu -\lambda \right) \Gamma \left( 1+\lambda \right) }%
\int_{\alpha }^{\infty }\frac{\cosh ^{-2\mu }sds}{\left( \sinh s-\sinh
\alpha \right) ^{-\lambda }}ds\text{.}  \label{cc1}
\end{equation}
\end{corollary}

Setting into (\ref{degr1}) $\rho =-k-1$, $k\in \mathbb{N}_{0}$, and
employing (\ref{Jac}), we arrive at another fractional type integral
representation of Ferrers functions generalizing (\ref{Ab2}):

\begin{corollary}
Let $\func{Re}\mu +k>\func{Re}\lambda >-1$. Then%
\begin{equation}
P_{\lambda -k}^{-\mu }\left( \tanh \alpha \right) =\frac{k!\cosh ^{k-\lambda
}\alpha }{\Gamma \left( k+\mu -\lambda \right) \Gamma \left( 1+\lambda
\right) }\int_{\alpha }^{\infty }\frac{e^{-\mu s}P_{k}^{\left( \mu ,-\mu
\right) }\left( \tanh s\right) ds}{\left( \sinh s-\sinh \alpha \right)
^{-\lambda }\cosh ^{k}s}\text{. }  \label{cc2}
\end{equation}
\end{corollary}

Note that integral representations obtained in this section, together with
analytic continuation, give rise to the following simple differentiation
formulas (which can be also obtained by using the differentiation formulas
for Gauss hypergeometric functions): as $n\in \mathbb{N}_{0}$, 
\begin{eqnarray*}
\frac{d^{n+1}}{dx^{n+1}}\frac{P_{\gamma -1}^{-n-\gamma }\left( x\right) }{%
\left( 1-x^{2}\right) ^{-\frac{n+\gamma }{2}}} &=&\frac{\left( -1\right)
^{n+1}2^{1-\gamma }}{\Gamma \left( \gamma \right) \left( 1-x^{2}\right)
^{1-\gamma }}\text{,} \\
\frac{d^{n+1}}{dx^{n+1}}\left( P_{n}^{-\mu }\left( x\right) \left( \frac{1-x%
}{1+x}\right) ^{\frac{\mu }{2}}\right) &=&\frac{2^{n+1}\left( 1-x\right)
^{\mu -n-1}}{\Gamma \left( \mu -n\right) \left( 1+x\right) ^{\mu +n+1}}\text{%
,} \\
\left( \left( 1-x^{2}\right) ^{\frac{3}{2}}\frac{d}{dx}\right) ^{n+1}\frac{%
P_{n}^{-\mu }\left( x\right) }{\left( 1-x^{2}\right) ^{\frac{n}{2}}} &=&%
\frac{\left( -1\right) ^{n}\left( 1-x\right) ^{\frac{\mu +3}{2}}}{\Gamma
\left( \mu -n\right) \left( 1+x\right) ^{\frac{\mu -3}{2}}}\text{,} \\
\left( \left( 1-x^{2}\right) ^{\frac{3}{2}}\frac{d}{dx}\right) ^{n+1}\frac{%
P_{n-\mu }^{-\mu }\left( x\right) }{(1-x^{2})^{\frac{n-\mu }{2}}} &=&\frac{%
2^{\mu }\Gamma \left( \mu +\frac{1}{2}\right) (1-x^{2})^{\mu }}{\sqrt{\pi }%
\Gamma \left( 2\mu -n\right) }\text{,}
\end{eqnarray*}

A representation of Ferrers functions in the form of a fractional type
integral containing a Gauss hypergeometric function can be derived by
exploiting the integral (\ref{int1}). As $\eta \left( s\right) =\tanh s$, on
inserting (\ref{int1}) into (\ref{Ab1}) and interchanging the order of
integration%
\begin{eqnarray*}
P_{\nu }^{-\mu }\left( \tanh \alpha \right) &=&\frac{2^{1+\nu }e^{\alpha \mu
}}{\Gamma \left( \mu -\nu \right) \Gamma \left( \lambda +1\right) \Gamma
\left( \nu -\lambda \right) }\int_{\alpha }^{\infty }\frac{J\left(
u,s\right) du}{\cosh ^{2}s\left( \tanh s-\tanh \alpha \right) ^{-\lambda }}%
\text{,} \\
J\left( u,s\right) &=&\int_{s}^{\infty }\frac{e^{-2u\mu }\cosh ^{2\nu }u}{%
\left( \tanh u-\tanh s\right) ^{\lambda +1-\nu }}du\text{.}
\end{eqnarray*}%
The integral $J\left( u,s\right) $ can be readily evaluated by making the
change $e^{-2u}=e^{-2s}z$ and using Euler's integral representation of the
Gauss hypergeometric function. Finally, 
\begin{eqnarray*}
\frac{2^{\nu }e^{-\alpha \mu }P_{\nu }^{-\mu }\left( \tanh \alpha \right) }{%
\Gamma ^{-1}\left( \mu -\lambda \right) \Gamma ^{-1}\left( 1+\lambda \right) 
} &=&\int_{\alpha }^{\infty }\frac{F\left( -\nu -\lambda -1,\mu -\nu ;\mu
-\lambda ;-e^{-2s}\right) ds}{\left( \tanh s-\tanh \alpha \right) ^{-\lambda
}e^{\left( 2\mu -\nu -\lambda -1\right) s}\cosh ^{\nu -\lambda +1}s}\text{,}
\\
\text{ }\func{Re}\mu &>&\func{Re}\lambda >-1\text{. }
\end{eqnarray*}

\section{ Series and integral relations connecting Ferrers functions of
different degrees and orders}

We commence with

\begin{theorem}
Let $\func{Re}\left( \mu +\nu \right) >-1$. Then on $-1<x<1$,%
\begin{equation}
\frac{P_{\nu }^{-\mu }\left( x\right) }{\Gamma \left( \nu -\mu +1\right) }=%
\frac{2^{-\mu }\left( 1-x\right) ^{^{\mu }}}{\Gamma \left( \nu +\mu
+1\right) }\sum_{n=0}^{\infty }\frac{\left( -2\mu \right) _{n}\left( -\nu
\right) _{n}}{n!}\left( \frac{1+x}{1-x}\right) ^{\frac{n}{2}}P_{\nu -\mu
}^{\mu -n}\left( x\right)  \label{con11}
\end{equation}
\end{theorem}

\begin{proof}
Suppose first that $\left( \func{Re}\gamma ,\func{Re}\sigma \right) $ is an
interior point of the triangle $\Delta $ in the $\left( \func{Re}\gamma ,%
\func{Re}\sigma \right) $-plane with $\Delta $ described by the inequalities 
$\func{Re}\gamma >0$, $\func{Re}\sigma >-1$, $\func{Re}\left( \sigma +\gamma
\right) >0$, and $\func{Re}\left( \sigma +2\gamma \right) <1$. By using the
absolutely convergent binomial series 
\[
2^{\varkappa }\sinh ^{\varkappa }\frac{t}{2}=\sum_{n=0}^{\infty }\frac{%
\left( -\varkappa \right) _{n}e^{\left( \frac{\varkappa }{2}-n\right) t}}{n!}%
\text{, }\varkappa >0\text{, }t\geq 0\text{,} 
\]%
we rewrite (\ref{q2}) in the form%
\[
P_{-\gamma }^{-\sigma -\gamma }\left( \tanh \alpha \right) =\frac{%
2^{-2\sigma -3\gamma }\cosh ^{\gamma }\alpha }{\Gamma \left( \gamma \right)
\Gamma \left( \sigma +1\right) }\lim_{\varepsilon \rightarrow
+0}\int_{\varepsilon }^{\frac{1}{\varepsilon }}\sum_{n=0}^{\infty }\frac{%
\left( -2\sigma -2\gamma \right) _{n}e^{\left( \sigma +\gamma -n\right) t}}{%
n!\left[ \sinh \frac{t}{2}\cosh \left( \frac{t}{2}+\alpha \right) \right]
^{\sigma +2\gamma }}dt\text{,} 
\]%
Since the above series convergent absolutely and uniformly on $0<\varepsilon
\leq t\leq 1/\varepsilon $, one is allowed to integrate term-by-term. By
exploiting the integral representation (\ref{q1+}) and denoting $\gamma
=-\nu $ and $\sigma +\gamma =\mu $, we obtain the expression 
\begin{equation}
\frac{\Gamma \left( \nu +\mu +1\right) P_{\nu }^{-\mu }\left( \tanh \alpha
\right) }{2^{-\mu }\Gamma \left( \nu -\mu +1\right) \cosh ^{-\mu }\alpha }%
=\sum_{n=0}^{\infty }\frac{\left( -2\mu \right) _{n}\left( -\nu \right) _{n}%
}{n!e^{\left( \mu -n\right) \alpha }}P_{\nu -\mu }^{\frac{\mu }{2}-n}\left(
\tanh \alpha \right) -\lim_{\varepsilon \rightarrow +0}P_{\varepsilon }\text{%
,}  \label{ddd}
\end{equation}%
where%
\begin{eqnarray}
\func{Re}\left( \mu -\nu \right) &<&1\text{, }\func{Re}\left( \mu +\nu
\right) >-1\text{, }\func{Re}\nu <0\text{, }\func{Re}\mu >0\text{,}
\label{range} \\
\frac{P_{\varepsilon }}{P\left( \alpha \right) } &=&\sum_{n=0}^{\infty }%
\frac{\left( -2\mu \right) _{n}}{n!}V_{n}\left( \alpha ,\varepsilon \right) 
\text{, }P\left( \alpha \right) =\frac{2^{\nu -\mu }\cosh ^{-\mu -\nu
}\alpha }{\Gamma \left( \nu -\mu +1\right) \Gamma \left( -\nu \right) }\text{%
,}  \nonumber \\
V_{n}\left( \alpha ,\varepsilon \right) &=&\left( \int_{0}^{\varepsilon
}+\int_{\frac{1}{\varepsilon }}^{\infty }\right) \frac{e^{\left( \mu
-n\right) t}dt}{\left[ \sinh \frac{t}{2}\cosh \left( \frac{t}{2}+\alpha
\right) \right] ^{\mu -\nu }}\text{,}  \nonumber \\
\left\vert V_{n}\left( \alpha ,\varepsilon \right) \right\vert &\leq
&V\left( \varepsilon \right) =\left( \int_{0}^{\varepsilon }+\int_{\frac{1}{%
\varepsilon }}^{\infty }\right) \frac{e^{t\func{Re}\mu }dt}{\left[ \sinh 
\frac{t}{2}\cosh \left( \frac{t}{2}+\alpha \right) \right] ^{\func{Re}\left(
\mu -\nu \right) }}\text{.}  \nonumber
\end{eqnarray}%
\ The inequality $\left\vert P_{\varepsilon }\left( \alpha \right)
\right\vert \leq V\left( \varepsilon \right) \left\vert P\left( \alpha
\right) \right\vert \sum_{n=0}^{\infty }\left( -2\mu \right) _{n}/n!$
manifests that the second term in (\ref{ddd}) vanishes because the integrand
in $V\left( \varepsilon \right) $ is an integrable function. Now, after the
change $\tanh \alpha =x$, we obtain (\ref{con11}) which is valid for $\left( 
\func{Re}\mu ,\func{Re}\nu \right) $ belonging to the triangle (\ref{range}%
). By taking into account an asymptotic formula%
\begin{eqnarray}
\Gamma \left( \lambda -\rho \right) P_{\tau }^{-\lambda }\left( \tanh \alpha
\right) &=&\frac{\Gamma \left( \lambda -\rho \right) }{\Gamma \left( \lambda
+1\right) }e^{-\lambda \alpha }\left( 1+O\left( \frac{1}{\lambda }\right)
\right)  \nonumber \\
&=&\frac{e^{-\lambda \alpha }}{\lambda ^{1+\rho }}\left( 1+O\left( \frac{1}{%
\lambda }\right) \right) \text{,}  \label{as1}
\end{eqnarray}%
which follows on from (\ref{1}) and (\ref{As-gam}) as $\left\vert \lambda
\right\vert \rightarrow \infty $ in the sector $\left\vert \arg \lambda
\right\vert \leq \delta <\pi $, one can infer that the series in (\ref{con11}%
) converges absolutely and uniformly with respect to $\mu $ and $\nu $ on
any bounded region of $\mathbb{C}^{2}$ in which $\func{Re}\left( \mu +\nu
\right) >-1$. Then the right side of (\ref{con11}) is an analytic function $%
\mu $ and $\nu $ on such a region, and the theorem is proven by virtue of
analytic continuation.
\end{proof}

\begin{corollary}
As $\ x\in \left( -1,1\right) $, $\mu ,\nu \in \mathbb{C}$, and $l,k\in 
\mathbb{N}_{0}$, there are the connection relations%
\begin{eqnarray}
\frac{P_{k}^{-\mu }\left( x\right) }{\Gamma \left( k-\mu +1\right) } &=&%
\frac{2^{-\mu }\left( 1-x\right) ^{^{\mu }}k!}{\Gamma \left( k+\mu +1\right) 
}\sum_{n=0}^{k}\frac{\left( -1\right) ^{n}\left( -2\mu \right) _{n}}{%
n!\left( k-n\right) !}\left( \frac{1+x}{1-x}\right) ^{\frac{n}{2}}P_{k-\mu
}^{\mu -n}\left( x\right) \text{,}  \label{col1} \\
\frac{P_{\nu }^{-\frac{l}{2}}\left( x\right) }{\Gamma \left( \nu -\frac{l}{2}%
+1\right) } &=&\frac{2^{-\frac{l}{2}}\left( 1-x\right) ^{\frac{l}{2}}l!}{%
\Gamma \left( \nu +\frac{l}{2}+1\right) }\sum_{n=0}^{l}\frac{\left(
-1\right) ^{n}\left( -\nu \right) _{n}}{n!\left( l-n\right) !}\left( \frac{%
1+x}{1-x}\right) ^{\frac{n}{2}}P_{\nu -\frac{l}{2}}^{\frac{l}{2}-n}\left(
x\right) \text{.}  \label{col2}
\end{eqnarray}%
\ 
\end{corollary}

The restriction on $\mu $ and $\nu $ in the above corollary was again
discarded due to analytic continuation.

Another connection relation can be derived by employing the absolutely and
uniformly convergent on $y\geq 0$ series 
\begin{equation}
e^{-\mu y}=\frac{\left( 1+\sqrt{1-\frac{1}{\cosh ^{2}\frac{y}{2}}}\right)
^{-2\mu }}{\cosh ^{2\mu }\frac{y}{2}}=\sum_{n=0}^{\infty }\frac{%
A_{n}^{\left( \mu \right) }}{\cosh ^{2n+2\mu }\frac{y}{2}}\text{, }
\label{e-ch}
\end{equation}%
where $\mu \in \mathbb{C}$ and 
\[
\text{ }A_{n}^{\left( \mu \right) }=\mu \frac{\left( 2\mu +n+1\right) _{n-1}%
}{2^{2n+2\mu -1}n!} 
\]%
are Taylor-Maclaurin coefficients of the function $\left( 1+\sqrt{1-x}%
\right) ^{-2\mu }$ found with Lagrange's expansion by noting that $y\left(
x\right) =1+\sqrt{1-x}$ is a solution of the equation $y=a-x/y$ as $a=2$.

\begin{theorem}
Let $\mu ,\nu \in \mathbb{C}$. Then on $0<x<1$,%
\begin{eqnarray}
\frac{P_{\nu }^{-\mu }\left( x\right) }{\left( 1+x\right) ^{\mu }} &=&\mu
\sum_{n=0}^{\infty }\frac{\left( 2\mu +n+1\right) _{n-1}\left( \mu -\nu
\right) _{n}}{2^{3n+3\mu -1}n!\left( 1-x^{2}\right) ^{-\frac{n}{2}}}P_{n+\mu
-\nu -1}^{-\mu -n}\left( x\right) \text{,}  \label{ee1} \\
\frac{P_{\nu }^{-2\mu }\left( x\right) }{\left( 1-x^{2}\right) ^{\frac{\mu }{%
2}}} &=&\frac{\mu }{2^{\mu }}\sum_{n=0}^{\infty }\frac{\Gamma \left( n+\mu
\right) \left( 2\mu -\nu \right) _{2n}}{2^{n}\Gamma \left( n+2\mu +1\right)
n!}\frac{P_{n+\mu -\nu -1}^{-n-\mu }\left( x\right) }{\left( 1-x^{2}\right)
^{-\frac{n}{2}}}\text{.}  \label{ee2}
\end{eqnarray}
\end{theorem}

\begin{proof}
When $\alpha >0$ and $\func{Re}\mu >\func{Re}\nu >0$, inserting the above
series (\ref{e-ch}) with $y=t+2\alpha $ into (\ref{q1+}) and integrating
term-by-term lead to a series of integrals. On evaluating arising integrals
with (\ref{q2}) one obtains 
\begin{equation}
P_{\nu }^{-\mu }\left( \tanh \alpha \right) =e^{\mu \alpha
}\sum_{n=0}^{\infty }A_{n}^{\left( \mu \right) }\frac{\Gamma \left( n+\mu
-\nu \right) P_{n+\mu -\nu -1}^{-\mu -n}\left( \tanh \alpha \right) }{2^{\mu
+n}\Gamma \left( \mu -\nu \right) \cosh ^{n+\mu }\alpha }\text{.}
\label{con!!a}
\end{equation}%
Now, note that as $\mu =\lambda +n$ and $\nu =\gamma +n$, $n\in \mathbb{N}%
_{0}$, the asymptotic formula for Gauss hypergeometric functions with large
parameters \cite{NIST}, \cite{Olde} shows the hypergeometric function in the
definition (\ref{Hyper1}) to be a bounded quantity when $\alpha >0$, $%
\lambda $, and $\gamma $ are fixed. Then, 
\begin{equation}
\left\vert \Gamma \left( \lambda +n+1\right) P_{\gamma +n}^{-\lambda
-n}\left( \tanh \alpha \right) \right\vert \leq \frac{B\left( \alpha
,\lambda ,\gamma \right) }{2^{n}\cosh ^{n}\alpha }\text{, }0<\alpha <\infty 
\text{.}  \label{est}
\end{equation}%
The above estimate manifests that series in (\ref{con!!a}) converges
absolutely and uniformly with respect to $\left( \mu ,\nu \right) $
belonging to any bounded region of $\mathbb{C}^{2}$, that is, restrictions
on the parameters $\mu $ and $\nu $ can be discarded due to analytic
continuation. Finally, on changing $\tanh \alpha =x$ the proof of (\ref{ee1}%
) is completed. The connection formula (\ref{ee2}) is proved in the same
manner by inserting (\ref{e-ch}) into (\ref{Ab2}) and evaluating arising
integrals with (\ref{cc1}).
\end{proof}

By employing (\ref{Jac1}), we have

\begin{corollary}
Let $\mu \in \mathbb{C}$ and $k\in \mathbb{N}_{0}$. Then on $0\leq x<1$,
\end{corollary}

\begin{equation}
\frac{P_{k+\mu }^{-\mu }\left( x\right) }{\left( 1+x\right) ^{\mu }}=\frac{%
\mu k!}{2^{3\mu -1}}\sum_{n=0}^{k}\frac{\left( -1\right) ^{n}\left( 2\mu
+n+1\right) _{n-1}P_{k-n}^{-n-\mu }\left( x\right) }{2^{3n}\left( k-n\right)
!n!\left( 1-x^{2}\right) ^{-\frac{n}{2}}}\text{,}  \label{con!3a}
\end{equation}%
and on $-1<x<1$,%
\begin{equation}
\frac{P_{k+2\mu }^{-2\mu }\left( x\right) }{\left( 1-x^{2}\right) ^{\frac{%
\mu }{2}}}=\frac{\mu k!}{2^{\mu }}\sum_{n=0}^{\left[ \frac{k}{2}\right] }%
\frac{\Gamma \left( n+\mu \right) \left( 1-x^{2}\right) ^{\frac{n}{2}%
}P_{k+\mu -n}^{-n-\mu }\left( x\right) }{2^{n}\Gamma \left( n+2\mu +1\right)
\left( k-2n\right) !n!}\text{.}  \label{newA}
\end{equation}

One can obtain additional similar relations as follows.

\begin{theorem}
Let $\sigma ,\gamma ,\nu ,\mu \in \mathbb{C}$. Then on $0<x<1$, 
\begin{equation}
P_{\gamma -1}^{-\sigma -\gamma }\left( x\right) =\sum_{n=0}^{\infty }\frac{%
\left( -1\right) ^{n}\left( 2n\right) !\left( \gamma -1\right) _{n}}{%
2^{n-2}\left( n!\right) ^{2}\left( 1-x^{2}\right) ^{-\frac{n}{2}}}%
P_{n+\gamma -2}^{-\sigma -\gamma -n}\left( x\right) \text{;}  \label{sh-ch1}
\end{equation}%
\begin{eqnarray}
P_{\lambda -\mu }^{-\mu }\left( x\right) &=&\frac{\Gamma \left( \mu +\frac{3%
}{2}\right) }{\Gamma \left( \mu +1\right) }\sum_{n=0}^{\infty }\frac{\left(
2n\right) !\left( 2\mu -\lambda \right) _{2n}P_{n-\lambda +\mu -\frac{3}{2}%
}^{-n-\mu -\frac{1}{2}}\left( x\right) }{\left( -1\right) ^{n}2^{3n-\frac{1}{%
2}}n!\left( \mu +1\right) _{n}\left( 1-x^{2}\right) ^{-\frac{2n-1}{4}}}\text{%
,}  \label{sh-ch2} \\
-\mu -\frac{3}{2} &\notin &\mathbb{N}_{0}\text{.}  \nonumber
\end{eqnarray}
\end{theorem}

\begin{proof}
Let $\alpha >0$, $2\func{Re}\mu >\func{Re}\lambda >1$, $\func{Re}\gamma >1$,
and $\func{Re}\sigma >-1$. Then integrating the integral representations (%
\ref{Ab7}) and (\ref{cc1}) by parts, we get 
\begin{eqnarray*}
P_{\gamma -1}^{-\sigma -\gamma }\left( \tanh \alpha \right) &=&\frac{%
2^{2-\gamma }\cosh ^{\sigma +\gamma }\alpha }{\Gamma \left( \gamma -1\right)
\Gamma \left( \sigma +2\right) }\int_{\alpha }^{\infty }\frac{\left( \tanh
u-\tanh \alpha \right) ^{\sigma +1}\sinh u}{\cosh ^{2\gamma -1}u}du\text{,}
\\
P_{\lambda -\mu }^{-\mu }\left( \tanh \alpha \right) &=&\frac{2^{\mu
+1}\Gamma \left( \mu +\frac{3}{2}\right) \cosh ^{\mu -\lambda }\alpha }{%
\sqrt{\pi }\Gamma \left( 2\mu -\lambda \right) \Gamma \left( \lambda
+2\right) }\int_{\alpha }^{\infty }\frac{\cosh ^{-2\mu -2}u\sinh udu}{\left(
\sinh u-\sinh \alpha \right) ^{-\lambda -1}}\text{.}
\end{eqnarray*}%
Substituting the binomial series 
\[
\sinh u=\cosh u\sqrt{1-\frac{1}{\cosh ^{2}u}}=\sum_{n=0}^{\infty }\frac{%
\left( -1\right) ^{n}\left( 2n\right) !}{2^{2n}n!\cosh ^{2n-1}u}\text{, }u>0%
\text{,} 
\]%
and integrating term-by-term lead to the expansions (\ref{sh-ch1}) and (\ref%
{sh-ch2}) where $x=\tanh \alpha $. Finally, due to (\ref{est}) one can
ascertain that\ the theorem holds.
\end{proof}

\begin{corollary}
Let $\mu ,\sigma \in \mathbb{C}$ and $k\in \mathbb{N}_{0}$. Then on $0\leq
x<1$,%
\begin{equation}
P_{k-1}^{k-\sigma }\left( x\right) =k!\sum_{n=0}^{k}\frac{\left( 2n\right)
!\left( 1-x^{2}\right) ^{\frac{n}{2}}}{2^{n-2}\left( k-n\right) !\left(
n!\right) ^{2}}P_{k-n}^{k-n-\sigma -1}\left( x\right) \text{,}  \label{con!4}
\end{equation}%
If $-\mu -\frac{3}{2}\notin \mathbb{N}_{0}$, then on $-1<x<1$,%
\begin{equation}
P_{k+\mu }^{-\mu }\left( x\right) =\frac{\Gamma \left( \mu +\frac{3}{2}%
\right) k!}{\Gamma \left( \mu +1\right) }\sum_{n=0}^{\left[ \frac{k}{2}%
\right] }\frac{\left( 2n\right) !\left( 1-x^{2}\right) ^{\frac{2n-1}{4}%
}P_{k+\mu -n+\frac{1}{2}}^{-n-\mu -\frac{1}{2}}\left( x\right) }{\left(
-1\right) ^{n}2^{3n-\frac{1}{2}}\left( k-2n\right) !n!\left( \mu +1\right)
_{n}}\text{.}  \label{con!5}
\end{equation}
\end{corollary}

As $\func{Re}\sigma >-1$ and $\func{Re}\gamma >0$, a new integral connection
between Ferrers functions can be obtained by rewriting (\ref{Ab7}) in the
form%
\begin{eqnarray}
P_{\gamma -1}^{-\sigma -\gamma }\left( \tanh \alpha \right) &=&\frac{%
2^{1-\gamma }\cosh ^{\sigma +\gamma }\alpha }{\Gamma \left( \gamma \right)
\Gamma \left( \sigma +1\right) }\int_{-\infty }^{\infty }f\left( u,\alpha
\right) \frac{e^{2\epsilon u}}{\cosh ^{2\left( \gamma +\sigma \right) }u}du%
\text{,}  \label{P-new} \\
f\left( u,\alpha \right) &=&H\left( u-\alpha \right) e^{-2\epsilon u}\left(
\tanh u-\tanh \alpha \right) ^{\sigma }\cosh ^{2\sigma }u\text{,}  \nonumber
\end{eqnarray}%
where $\func{Re}\sigma <\func{Re}\epsilon <\func{Re}\left( \gamma +\sigma
\right) $ and $H\left( u\right) $ is the Heaviside unit function. Employ
Parseval's equation for Fourier transforms \cite{Tit}, 
\begin{eqnarray}
\int_{-\infty }^{\infty }g_{1}\left( u\right) g_{2}\left( u\right) du &=&%
\frac{1}{2\pi }\int_{-\infty }^{\infty }G_{1}\left( \omega \right)
G_{2}\left( -\omega \right) d\omega \text{,}  \label{Pars} \\
G_{m}\left( \omega \right) &=&\int_{-\infty }^{\infty }g_{m}\left( u\right)
e^{i\omega u}du\text{,}  \nonumber \\
g_{1}\left( u\right) &\in &L^{p}\left( -\infty ,\infty \right) \text{, }%
G_{2}\left( \omega \right) \in L^{p}\left( -\infty ,\infty \right) \text{, }%
1\leq p\leq 2\text{,}  \nonumber
\end{eqnarray}%
to transform (\ref{P-new}). On exploiting (\ref{Ab1}) and

\begin{eqnarray}
\int_{-\infty }^{\infty }\frac{e^{pu}du}{\cosh ^{q}u} &=&2^{q}\int_{0}^{%
\infty }\frac{t^{p+q-1}dt}{\left( t^{2}+1\right) ^{q}}=2^{q-1}\int_{0}^{%
\infty }\frac{s^{\frac{p+q}{2}-1}ds}{\left( s+1\right) ^{q}}  \nonumber \\
&=&\frac{2^{q-1}\Gamma \left( \frac{q+p}{2}\right) \Gamma \left( \frac{q-p}{2%
}\right) }{\Gamma \left( q\right) }\text{, }\func{Re}\left( q\pm p\right) >0%
\text{,}  \label{ep-ch}
\end{eqnarray}%
it results after the changes $z=\epsilon -i\frac{\omega }{2}$, $\sigma
+\gamma =\mu $, and $\gamma =-\nu $ in the connection formula%
\begin{eqnarray}
\frac{P_{\nu }^{-\mu }\left( \tanh \alpha \right) }{\cosh ^{\mu }\alpha } &=&%
\frac{\mathcal{A}}{2\pi i}\int\limits_{\func{Re}z=\epsilon }\frac{P_{\mu
+\nu }^{-z}\left( \tanh \alpha \right) \Gamma \left( \mu -z\right)
e^{-z\alpha }}{\Gamma ^{-1}\left( \mu +z\right) \Gamma ^{-1}\left( z-\mu
-\nu \right) }dz\text{,}  \label{con1} \\
\mathcal{A} &=&\frac{2^{\mu }}{\Gamma \left( -\nu \right) \Gamma \left( 2\mu
\right) }\text{, }  \nonumber
\end{eqnarray}%
where $\func{Re}\mu >0$, $\func{Re}\nu <0$, and $\max \left( -\func{Re}\mu 
\text{,}\func{Re}\left( \mu +\nu \right) \right) <\func{Re}\epsilon <\func{Re%
}\mu $.

The integral in (\ref{con1}) can be evaluated as $\alpha >0$ and $\alpha <0$
in the form of certain series by applying the residue theorem. This
evaluation and analytic continuation yield as $\tanh \alpha =x$,

\begin{theorem}
As $\nu ,\mu \in \mathbb{C}$ and $0<x<1$, 
\[
P_{\nu }^{-\mu }\left( x\right) =2^{\mu }\left( 1+x\right) ^{-\mu
}\sum_{n=0}^{\infty }\frac{\left( -1\right) ^{n}\left( 2\mu \right)
_{n}\left( -\nu \right) _{n}}{n!}\left( \frac{1-x}{1+x}\right) ^{\frac{n}{2}%
}P_{\mu +\nu }^{-\mu -n}\left( x\right) \text{,} 
\]%
If$\ $ $2\mu +\nu \notin \mathbb{Z}$ when $\nu \notin \mathbb{Z}$ and $%
-1<x<0 $, 
\begin{eqnarray*}
\frac{2^{-\mu }P_{\nu }^{-\mu }\left( x\right) }{\left( 1-x^{2}\right) ^{-%
\frac{\mu }{2}}} &=&\sum_{n=0}^{\infty }\frac{\left( 2\mu \right) _{n}\Gamma
\left( -\nu -2\mu -n\right) }{\left( -1\right) ^{n}\Gamma \left( -\nu
\right) n!}\left( \frac{1+x}{1-x}\right) ^{\frac{n+\mu }{2}}P_{\nu +\mu
}^{\mu +n}\left( x\right) \\
&&+\sum_{n=0}^{\infty }\frac{\left( -\nu \right) _{n}\Gamma \left( 2\mu +\nu
-n\right) }{\left( -1\right) ^{n}n!\Gamma \left( 2\mu \right) }\left( \frac{%
1+x}{1-x}\right) ^{\frac{n-\mu -\nu }{2}}P_{\nu +\mu }^{n-\mu -\nu }\left(
x\right) \text{,}
\end{eqnarray*}
\end{theorem}

\begin{corollary}
As$\ x\in \left( -1,1\right) $, $k,l\in \mathbb{N}_{0}$, and $\mu ,\nu \in 
\mathbb{C}$ , 
\begin{eqnarray}
\frac{2^{-\mu }P_{k}^{-\mu }\left( x\right) }{k!\left( 1+x\right) ^{-\mu }}
&=&\sum_{j=0}^{k}\frac{\left( 2\mu \right) _{j}}{\left( k-j\right) !j!}%
\left( \frac{1-x}{1+x}\right) ^{\frac{j}{2}}P_{k+\mu }^{-j-\mu }\left(
x\right) \text{, }  \label{qq1} \\
\frac{2^{\frac{l}{2}}P_{\nu }^{\frac{l}{2}}\left( x\right) }{l!\left(
1+x\right) ^{\frac{l}{2}}} &=&\sum_{j=0}^{l}\frac{\left( -\nu \right) _{j}}{%
\left( l-j\right) !j!}\left( \frac{1-x}{1+x}\right) ^{\frac{j}{2}}P_{\nu -%
\frac{l}{2}}^{\frac{l}{2}-j}\left( x\right) \text{.}  \label{qq2}
\end{eqnarray}
\end{corollary}

\bigskip The integral connection formula 
\begin{eqnarray*}
\frac{P_{\lambda -\mu }^{-\mu }\left( \tanh \alpha \right) }{\cosh ^{\mu
}\alpha } &=&\frac{\mathcal{A}_{0}}{2\pi i}\int\limits_{\func{Re}z=\epsilon }%
\frac{\Gamma \left( z-\lambda \right) P_{\lambda }^{-z}\left( \tanh \alpha
\right) }{\Gamma ^{-1}\left( \mu -\frac{z}{2}\right) \Gamma ^{-1}\left( \mu +%
\frac{z}{2}\right) }dz\text{,} \\
\mathcal{A}_{0} &=&\frac{2^{\mu }}{\Gamma \left( 2\mu -\lambda \right)
\Gamma \left( \mu \right) }\text{, }\func{Re}\lambda >-1\text{,}\func{Re}\mu
>0\text{,} \\
-1 &<&\max \left( \func{Re}\lambda \text{,}-2\func{Re}\mu \right) <\epsilon
<2\func{Re}\mu \text{,}
\end{eqnarray*}%
is derived in the same manner from (\ref{cc1}) by exploiting (\ref{Ab2}) and
(\ref{ep-ch}). Evaluating the integral by the residue theorem and using
analytic continuation, we have

\begin{theorem}
As $\lambda ,\mu \in \mathbb{C}$ and $x>0$, 
\[
P_{\lambda -\mu }^{-\mu }\left( x\right) =\frac{2^{3\mu }\Gamma \left( \mu +%
\frac{1}{2}\right) }{\sqrt{\pi }(1-x^{2})^{\frac{\mu }{2}}}%
\sum_{n=0}^{\infty }\frac{\left( 2\mu -\lambda \right) _{2n}\left( 2\mu
\right) _{n}}{\left( -1\right) ^{n}n!}P_{\lambda }^{-2\mu -2n}\left(
x\right) \text{.} 
\]
\end{theorem}

\begin{corollary}
Let $\mu \in \mathbb{C}$ and $k,\lambda \in \mathbb{N}_{0}$. Then as $-1<x<1$%
,%
\begin{equation}
P_{k+\mu }^{-\mu }\left( x\right) =\frac{2^{3\mu }k!\Gamma \left( \mu +\frac{%
1}{2}\right) }{\sqrt{\pi }(1-x^{2})^{\frac{\mu }{2}}}\sum_{n=0}^{\left[ 
\frac{k}{2}\right] }\frac{\left( 2\mu \right) _{n}P_{k+2\mu }^{-2\mu
-2n}\left( x\right) }{\left( -1\right) ^{n}n!\left( k-2n\right) !}
\label{new12}
\end{equation}%
and 
\begin{equation}
P_{\lambda +k}^{k}\left( x\right) =\left( -2\right) ^{-k}k!(1-x^{2})^{\frac{k%
}{2}}\sum_{n=0}^{2k}\frac{\left( -2k-\lambda \right) _{2n}}{n!\left(
2k-n\right) !}P_{\lambda }^{2k-2n}\left( x\right) \text{.}  \label{new11}
\end{equation}%
\ If $\ \lambda \in \mathbb{C}\backslash \mathbb{N}_{0}$, then (\ref{new11})
holds as $0\leq x<1$.
\end{corollary}

The formula connecting Ferrers functions with products of Ferrers functions,
namely 
\begin{eqnarray}
P_{\nu }^{-\mu }\left( x\right) &=&\mathcal{B}\int_{-\infty }^{\infty }\frac{%
P_{\nu _{1}}^{-\mu _{1}-i\omega }\left( x\right) P_{\nu _{2}}^{-\mu
_{2}+i\omega }\left( x\right) d\omega }{\Gamma ^{-1}\left( \mu _{1}-\nu
_{1}+i\omega \right) \Gamma ^{-1}\left( \mu _{2}-\nu _{2}-i\omega \right) }%
\text{,}  \label{ppp} \\
\mathcal{B} &\mathcal{=}&\frac{\Gamma \left( 1+\nu _{1}\right) \Gamma \left(
1+\nu _{2}\right) }{2\pi \Gamma \left( 1+\nu \right) \Gamma \left( \mu -\nu
\right) }\text{,}  \nonumber
\end{eqnarray}%
where%
\begin{equation}
\mu =\mu _{1}+\mu _{2}\text{, }\nu =\nu _{1}+\nu _{2}\text{, }\nu >-1\text{, 
}-\nu _{k}\notin \mathbb{N}\text{, }\func{Re}\left( \nu _{k}-\mu _{k}\right)
\notin \mathbb{N}_{0}\text{,}  \label{ppp2}
\end{equation}%
is obtained from (\ref{q1+}) by using Parseval's equation for Fourier
integral transforms, making the change $\tanh \alpha =x$, and taking such $%
\mu _{k}\ $and $\nu _{k}$ that $\func{Re}\mu _{k}>\func{Re}\nu _{k}>-1$. One
can see from (\ref{as1}) that for any $\mu _{k}$ the integrand in (\ref{ppp}%
) is $O\left( \omega ^{-2-\nu }\right) $ as $\left\vert \omega \right\vert
\rightarrow \infty $.\ If $\func{Re}\left( \nu _{k}-\mu _{k}\right) \notin 
\mathbb{N}_{0}$, the integrand is an analytic function of parameters $\nu
_{k}$ and $\mu _{k}$, and then integral (\ref{ppp}) holds under conditions (%
\ref{ppp2}) due to analytic continuation.

On evaluating integral (\ref{ppp}) by the residue theorem and denoting $\nu
_{1}=\lambda $ we obtain after analytic continuation

\begin{equation}
P_{\nu }^{-\mu }\left( x\right) =\frac{\Gamma \left( 1+\lambda \right)
\Gamma \left( 1+\nu -\lambda \right) }{\Gamma \left( 1+\nu \right) }%
\sum_{n=0}^{\infty }\frac{\left( \mu -\nu \right) _{n}}{\left( -1\right)
^{n}n!}P_{\lambda }^{n-\lambda }\left( x\right) P_{\nu -\lambda }^{-n-\mu
+\lambda }\left( x\right) \text{.}  \label{G0}
\end{equation}%
The above expression can be transformed by some manipulations to the form
established by Campos \cite{Campos} whose proof is quite different.

Employing Parseval's equations for cosine and sine Fourier integral
transforms again leads from (\ref{q1+}) to similar integral relations. In
particular, as $\nu ,\mu \in \mathbb{R}$, $\mu _{1}=\mu _{2}=\mu /2$, $\nu
_{1}=\nu _{2}=\nu /2$, and $\mu >\nu >-1$,

\begin{eqnarray*}
P_{\nu }^{-\mu }\left( \tanh \alpha \right) &=&\frac{2\Gamma ^{2}\left( 1+%
\frac{\nu }{2}\right) }{\pi \Gamma \left( \mu -\nu \right) \Gamma \left(
1+\nu \right) }\int_{0}^{\infty }\left( \func{Re}\frac{e^{-i\omega \alpha
}P_{\frac{\nu }{2}}^{-\frac{\mu }{2}+i\omega }\left( \tanh \alpha \right) }{%
\Gamma ^{-1}\left( \frac{\mu -\nu }{2}-i\omega \right) }\right) ^{2}d\omega
\\
&=&\frac{2\Gamma ^{2}\left( 1+\frac{\nu }{2}\right) }{\pi \Gamma \left( \mu
-\nu \right) \Gamma \left( 1+\nu \right) }\int_{0}^{\infty }\left( \func{Im}%
\frac{e^{-i\omega \alpha }P_{\frac{\nu }{2}}^{-\frac{\mu }{2}+i\omega
}\left( \tanh \alpha \right) }{\Gamma ^{-1}\left( \frac{\mu -\nu }{2}%
-i\omega \right) }\right) ^{2}d\omega \\
&=&\frac{\Gamma ^{2}\left( 1+\frac{\nu }{2}\right) }{\pi \Gamma \left( \mu
-\nu \right) \Gamma \left( 1+\nu \right) }\int_{0}^{\infty }\left\vert \frac{%
P_{\frac{\nu }{2}}^{-\frac{\mu }{2}+i\omega }\left( \tanh \alpha \right) }{%
\Gamma ^{-1}\left( \frac{\mu -\nu }{2}-i\omega \right) }\right\vert
^{2}d\omega \text{.}
\end{eqnarray*}

\section{Differential relations between Ferrers functions of different
degrees and orders}

By denoting $\rho =\nu -n$ and $\lambda =n-1$, $n\in \mathbb{N}$, (\ref%
{degr1}) is written as \ 
\begin{eqnarray}
\frac{P_{\nu }^{-\mu }\left( \tanh \alpha \right) }{\cosh ^{-\nu }\alpha }
&=&\frac{\left( \mu -\nu \right) _{n}}{\left( n-1\right) !}\int_{\alpha
}^{\infty }\frac{P_{\nu -n}^{-\mu }\left( \tanh s\right) \cosh ^{\nu -n+1}s}{%
\left( \sinh s-\sinh \alpha \right) ^{1-n}}ds  \label{N1} \\
\func{Re}\left( \mu -\nu \right) &>&-1\text{. }  \nonumber
\end{eqnarray}%
Therefore, 
\begin{equation}
\frac{\left( \mu -\nu \right) _{n}P_{n-\nu -1}^{-\mu }\left( \tanh \alpha
\right) }{\left( -1\right) ^{n}\cosh ^{n-\nu }\alpha }=\left( \frac{1}{\cosh
\alpha }\frac{d}{d\alpha }\right) ^{n}\frac{P_{\nu }^{-\mu }\left( \tanh
\alpha \right) }{\cosh ^{-\nu }\alpha }\text{.}  \label{Dif1}
\end{equation}%
Rewriting (\ref{N1}) in the form%
\begin{eqnarray*}
\frac{\cosh ^{\nu }\alpha }{\omega ^{1-n}\left( \alpha \right) }P_{\nu
}^{-\mu }\left( \tanh \alpha \right) &=&\frac{\left( \mu -\nu \right) _{n}}{%
\left( n-1\right) !}\int_{\alpha }^{\infty }\frac{P_{\nu -n}^{-\mu }\left(
\tanh s\right) \cosh ^{\nu -n+1}s}{\omega ^{n-1}\left( s\right) \left(
\omega \left( s\right) -\omega \left( \alpha \right) \right) ^{1-n}}ds\text{,%
} \\
\omega \left( s\right) &=&-\frac{1}{\sinh s+z}\text{, }z\notin \left(
-\infty ,-\sinh \alpha \right) \text{,}
\end{eqnarray*}%
we obtain%
\begin{equation}
\frac{\left( -1\right) ^{n}\left( \mu -\nu \right) _{n}P_{n-\nu -1}^{-\mu
}\left( \tanh \alpha \right) }{\left( z+\sinh \alpha \right) ^{-n-1}\cosh
^{n-\nu }\alpha }=\left( \frac{\left( z+\sinh \alpha \right) }{\cosh \alpha }%
^{2}\frac{d}{d\alpha }\right) ^{n}\frac{\cosh ^{\nu }\alpha P_{\nu }^{-\mu
}\left( \tanh \alpha \right) }{\left( z+\sinh \alpha \right) ^{n-1}}\text{.}
\label{Dif1-1}
\end{equation}%
\ Now, if $-1<x<1$ and $x+z\sqrt{1-x^{2}}\neq 0$, $z\in \mathbb{C}$, then%
\begin{equation}
\frac{P_{n-\nu -1}^{-\mu }\left( x\right) }{\left( 1-x^{2}\right) ^{\frac{%
\nu -n}{2}}}=\frac{\left( -1\right) ^{n}}{\left( \mu -\nu \right) _{n}}%
\left( \left( 1-x^{2}\right) ^{\frac{3}{2}}\frac{d}{dx}\right) ^{n}\frac{%
P_{\nu }^{-\mu }\left( x\right) }{\left( 1-x^{2}\right) ^{\frac{\nu }{2}}}%
\text{,}  \label{Dif2}
\end{equation}%
\begin{eqnarray}
\frac{q^{n+1}\left( x,z\right) P_{n-\nu -1}^{-\mu }\left( x\right) }{\left(
-1\right) ^{n}\left( 1-x^{2}\right) ^{\frac{\nu +1}{2}}} &=&\left( \frac{%
q^{2}\left( x,z\right) }{\left( 1-x^{2}\right) ^{-\frac{1}{2}}}\frac{d}{dx}%
\right) ^{n}\frac{q^{1-n}\left( x,z\right) P_{\nu }^{-\mu }\left( x\right) }{%
\left( \mu -\nu \right) _{n}\left( 1-x^{2}\right) ^{\frac{\nu -n+1}{2}}}%
\text{,}  \label{Dif2a} \\
q\left( x,z\right) &=&x+z\sqrt{1-x^{2}}\text{.}  \nonumber
\end{eqnarray}%
In the same way, as $\mu =\sigma -n$ and $\lambda =n-1$, it follows from (%
\ref{Con}) that\bigskip\ if $-1<x<1$ and $x+z\neq 0$, then%
\begin{eqnarray}
\frac{\left( -1\right) ^{n}P_{\nu }^{n-\sigma }\left( x\right) }{\left(
1-x^{2}\right) ^{\frac{n-\sigma }{2}}} &=&\frac{d^{n}}{dx^{n}}\frac{P_{\nu
}^{-\sigma }\left( x\right) }{\left( 1-x^{2}\right) ^{-\frac{\sigma }{2}}}%
\text{,}  \label{Dif3} \\
\frac{\left( x+z\right) ^{n+1}P_{\nu }^{n-\sigma }\left( x\right) }{\left(
-1\right) ^{n}\left( 1-x^{2}\right) ^{\frac{n-\sigma }{2}}} &=&\left( \left(
x+z\right) ^{2}\frac{d}{dx}\right) ^{n}\frac{P_{\nu }^{-\sigma }\left(
x\right) \left( 1-x^{2}\right) ^{\frac{\sigma }{2}}}{\left( x+z\right) ^{n-1}%
}\text{.}  \label{Dif3a}
\end{eqnarray}

The well-known (usually as $\sigma $ is an integer) elementary differential
relation (\ref{Dif3}) for $n=1$ also follows from the formula of the
derivative of Ferrers functions \cite{NIST}, \cite{PBM3} and then by
induction for any $n$. The differential relation (\ref{Dif2}) can be
obtained in the same manner as well (though surprisingly there is no source
of reference known to the author in the general case $n>1$). Readers can
find several similar differential recurrences (for $n=1$), including (\ref%
{Dif2}) and (\ref{Dif3}), in papers by Celeghini and del Olmo \cite{Celeg}
when $\nu ,%
\mu
$ are integers, and by Maier \cite{Maier1} when $\nu ,%
\mu
$ are arbitrary numbers. The differential relations (\ref{Dif2a}) and (\ref%
{Dif3a}) are not so evident and probably are new.

Combining or setting specific values of parameters, one can obtain from (\ref%
{Dif2}), (\ref{Dif2a}), (\ref{Dif3}), and (\ref{Dif3a}) various relations.
For example, as $z=\tan \varphi $ and $x=\cos \left( \varphi +\theta \right) 
$, $0<\varphi +\theta <\pi $, (\ref{Dif2a}) results in 
\[
\frac{\left( \mu -\nu \right) _{n}P_{n-\nu -1}^{-\mu }\left( \cos \left(
\theta +\varphi \right) \right) }{\cos ^{-n-1}\theta \sin ^{\nu +1}\left(
\theta +\varphi \right) }=\left( \cos ^{2}\theta \frac{d}{d\theta }\right)
^{n}\frac{P_{\nu }^{-\mu }\left( \cos \left( \theta +\varphi \right) \right) 
}{\cos ^{n-1}\theta \sin ^{\nu +1-n}\left( \theta +\varphi \right) }\text{.} 
\]

Another example is the relation obtained from (\ref{Dif3}) by making use of (%
\ref{F}), namely 
\begin{equation}
P_{\nu }^{n-\nu }\left( x\right) =\frac{\left( -1\right) ^{n}\left(
1-x^{2}\right) ^{\frac{n-\nu }{2}}}{2^{\nu }\Gamma \left( \nu +1\right) }%
\frac{d^{n}}{dx^{n}}\left( 1-x^{2}\right) ^{\nu }\text{,}  \label{ss}
\end{equation}%
that due to (\ref{Jac0}) is the classical Rodrigues formula for Jacobi
polynomials $P_{n}^{(\nu -n,\nu -n)}\left( x\right) $.

\section{Generating functions}

Ferrers functions $P_{\nu }^{-\mu }\left( x\right) $, $x\in \left(
1,1\right) $, can be analytically continued (by means of (\ref{Hyper1}) or (%
\ref{Hyper2})) to the complex $z$-plane with cuts connecting branch points $%
z=\pm 1$ and the point at infinity. We denote such analytic functions as $%
\widehat{P}_{\nu }^{-\mu }\left( z\right) $ to avoid some confusion between
associated Legendre functions $P_{\nu }^{-\mu }\left( z\right) $ and
continuations of Ferrers functions.

\begin{theorem}
Let $\mu ,\nu ,s\in \mathbb{C}$. Then, for $x\in \left( 1,1\right) $ there
are the generating functions:%
\begin{equation}
\sum_{n=0}^{\infty }\left( \mu -\nu \right) _{n}P_{n-\nu -1}^{-\mu }\left(
x\right) \frac{s^{n}}{n!}=\frac{\widehat{P}_{\nu }^{-\mu }\left( \frac{x-s}{%
\sqrt{1-2sx+s^{2}}}\right) }{\left( 1-2sx+s^{2}\right) ^{-\frac{\nu }{2}}}%
\text{,}  \label{gen1aa}
\end{equation}%
where $s\in \mathbb{C}$ if $\nu -\mu \in \mathbb{N}_{0}$ and $\left\vert
s\right\vert <1$ for $\nu -\mu \notin \mathbb{N}_{0}$. If $\left\vert
s\right\vert =1$ and $\nu -\mu \notin \mathbb{N}_{0}$, then (\ref{gen1aa})
is valid either as $\func{Re}\nu >1/2$ $\ $or as $-1/2<\func{Re}\nu \leq 1/2$
while $\arg s\neq $ $\pm \arccos x$;%
\begin{equation}
\sum_{n=0}^{\infty }\frac{P_{\nu }^{n-\mu }\left( x\right) }{\left(
1-x^{2}\right) ^{\frac{n-\mu }{2}}}\frac{s^{n}}{n!}=\frac{\widehat{P}_{\nu
}^{-\mu }\left( x-s\right) }{\left( 1-\left( x-s\right) ^{2}\right) ^{-\frac{%
\mu }{2}}}\text{, }  \label{gen2a}
\end{equation}%
where $s\in \mathbb{C}$ if $\ \nu $,$\mu \in \mathbb{Z}$; if $\ \nu $,$\mu
\notin \mathbb{Z}$, then either $\left\vert s\right\vert \leq 1-\left\vert
x\right\vert $ as $\func{Re}\mu >0$ or $\left\vert s\right\vert
<1-\left\vert x\right\vert $ as $\func{Re}\mu \leq 0$; if $\nu \in \mathbb{Z}
$ and $\mu \notin \mathbb{Z}$, then either $\left\vert s\right\vert \leq 1+x$
as $\func{Re}\mu >0$ or $\left\vert s\right\vert <1+x$ as $\func{Re}\mu \leq
0$; if $\mu \in \mathbb{Z}$ and $\nu \notin \mathbb{Z}$, then either $%
\left\vert s\right\vert \leq 1-x$ as $\func{Re}\mu >0$ or $\left\vert
s\right\vert <1-x$ as $\func{Re}\mu \leq 0$; \ 
\begin{equation}
\sum_{n=0}^{\infty }P_{\nu }^{n-\mu }\left( x\right) \frac{s^{n}}{n!}=\frac{%
\widehat{P}_{\nu }^{-\mu }\left( x-s\sqrt{1-x^{2}}\right) }{\left( 1+2s\frac{%
x}{\sqrt{1-x^{2}}}-s^{2}\right) ^{-\frac{\mu }{2}}}\text{, }  \label{gen3a}
\end{equation}%
where $s\in \mathbb{C}$ if $\ \nu $,$\mu \in \mathbb{Z}$, if $\ \nu $,$\mu
\notin \mathbb{Z}$, then either $\left\vert s\right\vert \leq \sqrt{%
1-\left\vert x\right\vert }/\sqrt{1+\left\vert x\right\vert }$ as $\mu >0$
or $\left\vert s\right\vert <\sqrt{1-\left\vert x\right\vert }/\sqrt{%
1+\left\vert x\right\vert }$ as Re$\mu \leq 0$; if $\nu \in \mathbb{Z}\ $and 
$\mu \notin \mathbb{Z}$, then either $\left\vert s\right\vert \leq \sqrt{1+x}%
/\sqrt{1-x}$ as $\func{Re}\mu >0$ or $\left\vert s\right\vert <\sqrt{1+x}/%
\sqrt{1-x}$ as $\func{Re}\mu \leq 0$; if $\mu \in \mathbb{Z}$ and $\ \nu
\notin \mathbb{Z}$, then either $\left\vert s\right\vert \leq \sqrt{1-x}/%
\sqrt{1+x}$ as $\mu >0$ or $\left\vert s\right\vert <\sqrt{1-x}/\sqrt{1+x}$
as $\mu \leq 0$.
\end{theorem}

\begin{proof}
On changing $\sinh \alpha =t$, $-\infty <t<\infty $,$\ $(\ref{Dif1}) leads
to the relation 
\begin{equation}
\frac{\left( -1\right) ^{n}\left( \mu -\nu \right) _{n}}{\left(
1+t^{2}\right) ^{\frac{n-\nu }{2}}}P_{n-\nu -1}^{-\mu }\left( \frac{t}{\sqrt{%
1+t^{2}}}\right) =\frac{d^{n}}{dt^{n}}\frac{P_{\nu }^{-\mu }\left( \frac{t}{%
\sqrt{1+t^{2}}}\right) }{\left( 1+t^{2}\right) ^{-\frac{\nu }{2}}}\text{,}
\label{Dif5}
\end{equation}%
which holds for $n\in \mathbb{N}$. Write the Taylor series 
\begin{equation}
\sum_{n=0}^{\infty }\frac{\left( -1\right) ^{n}\left( \mu -\nu \right) _{n}}{%
\left( 1+u^{2}\right) ^{\frac{n-\nu }{2}}}P_{n-\nu -1}^{-\mu }\left( \frac{u%
}{\sqrt{1+u^{2}}}\right) \frac{\left( t-u\right) ^{n}}{n!}=\frac{P_{\nu
}^{-\mu }\left( \frac{t}{\sqrt{1+t^{2}}}\right) }{\left( 1+t^{2}\right) ^{-%
\frac{\nu }{2}}}\text{.}  \label{Taylor}
\end{equation}%
As $P_{\nu }^{-\mu }\left( t/\sqrt{1+t^{2}}\right) $ is replaced by $%
\widehat{P}_{\nu }^{-\mu }\left( t/\sqrt{1+t^{2}}\right) $, (\ref{Taylor})
is valid by virtue of analytic continuation for complex $t$ lying within the
disk of convergence. Now, by changing%
\[
\frac{u}{\sqrt{1+u^{2}}}=x\text{, }\frac{u-t}{\sqrt{1+u^{2}}}=s\text{,} 
\]%
we obtain (\ref{gen1aa}). If $\nu -\mu \in \mathbb{N}_{0}$, then the series
in (\ref{gen1aa}) becomes a finite sum and $s\in \mathbb{C}$. If $\nu -\mu
\notin \mathbb{N}_{0}$, then the well-known asymptotics of Ferrers functions 
\cite{Erdelyi},\cite{Stegun}: for $\lambda \rightarrow \infty $ and $%
-1+\varepsilon <x<1-\varepsilon $, $\varepsilon >0$,%
\begin{eqnarray*}
P_{\lambda }^{-\mu }\left( x\right) &=&\frac{\Gamma \left( \lambda -\mu
+1\right) }{\Gamma \left( \lambda +\frac{3}{2}\right) }\left( \mathfrak{p}%
\left( \lambda ,x\right) +O\left( \frac{1}{\lambda }\right) \right) \text{,}
\\
\mathfrak{p}\left( \lambda ,x\right) &=&\sqrt{\frac{2}{\pi }}\frac{\cos
\left( \left( \lambda +\frac{1}{2}\right) \arccos x-\frac{\pi }{4}+\frac{\pi
\mu }{2}\right) }{\sqrt{1-x^{2}}}\text{,}
\end{eqnarray*}%
together with the asymptotics of gamma functions, leads to the estimate: as $%
n\rightarrow \infty $ 
\[
\left( \mu -\nu \right) _{n}P_{n-\nu -1}^{-\mu }\left( x\right) \frac{s^{n}}{%
n!}=\frac{s^{n}}{n^{\nu +\frac{1}{2}}}\left( \frac{\mathfrak{p}\left( n-\nu
-1,x\right) }{\Gamma \left( \mu -\nu \right) }+O\left( \frac{1}{n}\right)
\right) \text{,} 
\]%
which yields the restrictions imposed on $\left\vert s\right\vert $.

The generating functions (\ref{gen2a}) is derived in the same way as (\ref%
{gen1aa}) by employing (\ref{Dif3}). If $\nu ,\mu \in \mathbb{Z}$, then the
series in (\ref{gen2a}) is truncated and $s\in \mathbb{C}$. In the case when
at least one of the numbers $\nu $ and $\mu $ is not an integer, we suppose
by virtue of (\ref{1}) that $\func{Re}\nu >-1$. Invoking the asymptotic
formula (\ref{AsPp}) and using the relation \cite{NIST}

\[
\frac{\sin \pi \left( \nu -\sigma \right) }{\Gamma \left( \nu +\sigma
+1\right) }P_{\nu }^{\sigma }\left( x\right) =\frac{P_{\nu }^{-\sigma
}\left( x\right) \sin \pi \nu }{\Gamma \left( \nu -\sigma +1\right) }-\frac{%
\sin \pi \sigma P_{\nu }^{-\sigma }\left( -x\right) }{\Gamma \left( \nu
-\sigma +1\right) } 
\]%
manifest that as $n\rightarrow \infty $,%
\begin{eqnarray*}
\frac{P_{\nu }^{n-\mu }\left( x\right) }{\Gamma \left( n-\mu +\nu +1\right) }
&=&\frac{\left( \frac{1-x}{1+x}\right) ^{\frac{n}{2}}\mathfrak{C}_{n}\left(
x\right) \sin \pi \mu +\left( \frac{1+x}{1-x}\right) ^{\frac{n}{2}}\mathfrak{%
C}_{n}\left( -x\right) \sin \pi \nu }{n^{\nu +1}}\text{,} \\
\mathfrak{C}_{n}\left( x\right) &=&\alpha \left( x\right) +O\left( \frac{1}{n%
}\right) \text{, }\alpha \left( x\right) =\left( \frac{1-x}{1+x}\right) ^{%
\frac{\mu }{2}}\left( 1-x^{2}\right) ^{\frac{\nu }{2}}\text{,}
\end{eqnarray*}%
hence%
\[
\frac{P_{\nu }^{n-\mu }\left( x\right) s^{n}}{n!\left( 1-x^{2}\right) ^{%
\frac{n}{2}}}=\frac{s^{n}}{n^{\mu +1}}\left( \frac{\alpha \left( x\right)
+O\left( \frac{1}{n}\right) }{\left( 1+x\right) ^{n}}\sin \pi \mu +\frac{%
\alpha \left( -x\right) +O\left( \frac{1}{n}\right) }{\left( 1-x\right) ^{n}}%
\sin \pi \nu \right) \text{.} 
\]%
Analyzing the above relations, we arrive to the restrictions imposed on the
parameter $s$.

The generating function (\ref{gen3a}) follows on from (\ref{gen2a}) by a
change of a parameter. The proof is completed.
\end{proof}

The generating functions (\ref{gen1aa}) and (\ref{gen3a}) were obtained by
Truesdell \cite{Truesdell} but restrictions on the parameter $s$ were not
indicated, and besides, analytic continuation of Ferrers functions is
required as $s\in \mathbb{C}$. Truesdell also noted that the generating
function for Gegenbauer polynomials (\ref{g-geg}) is a special case of (\ref%
{gen1aa}).

Setting $\ x=0$ and $s=-t/\sqrt{1-t^{2}}$ into (\ref{gen1aa}) and using (\ref%
{zero}) yield new representations for Ferrers functions.

\begin{corollary}
Ferrers functions are represented in the forms%
\begin{eqnarray}
P_{\nu }^{-\mu }\left( t\right) &=&\frac{\sqrt{\pi }}{2^{\mu }}%
\sum_{n=0}^{\infty }\frac{\left( -1\right) ^{n}\left( \mu -\nu \right)
_{n}t^{n}\left( 1-t^{2}\right) ^{\frac{\nu -n}{2}}}{\Gamma \left( \frac{%
n+\mu -\nu +1}{2}\right) \Gamma \left( \frac{\mu +\nu -n+2}{2}\right) n!}%
\text{,}  \label{New0} \\
\text{ }\left\vert t\right\vert &<&\frac{\sqrt{2}}{2}\text{ if }\nu -\mu
\notin \mathbb{N}_{0}\text{.}  \nonumber
\end{eqnarray}%
If $\nu -\mu =l\in \mathbb{N}_{0}$,%
\begin{eqnarray}
P_{l+\mu }^{-\mu }\left( t\right) &=&\frac{\sqrt{\pi }l!}{2^{\mu }}%
\sum_{j=0}^{l}\frac{t^{n}\left( 1-t^{2}\right) ^{\frac{l+\mu -n}{2}}}{\left(
l-n\right) !\Gamma \left( \frac{n-l+1}{2}\right) \Gamma \left( \frac{2\mu
+l-n+2}{2}\right) n!}  \nonumber \\
&=&\frac{l!}{2^{\mu }}\sum_{j=0}^{\left[ \frac{l}{2}\right] }\frac{\left(
-1\right) ^{j}t^{l-2j}\left( 1-t^{2}\right) ^{j+\frac{\mu }{2}}}{%
2^{2j}j!\left( l-2j\right) !\Gamma \left( j+\mu +1\right) }\text{.}
\label{New1}
\end{eqnarray}
\end{corollary}

Setting $x=0$ and $s=-t$ into (\ref{gen3a}) and using (\ref{zero}) lead to

\begin{corollary}
Ferrers functions are represented on $-1<t<1$ in the form%
\begin{eqnarray}
\frac{P_{\nu }^{-\mu }\left( t\right) }{\left( 1-t^{2}\right) ^{-\frac{\mu }{%
2}}} &=&\frac{\sqrt{\pi }}{2^{\mu }}\sum_{n=0}^{\infty }\frac{\left(
-1\right) ^{n}t^{n}}{\Gamma \left( \frac{\mu -n-\nu +1}{2}\right) \Gamma
\left( \frac{\mu +\nu -n+2}{2}\right) n!}  \label{New00} \\
&=&-\frac{\sin \pi \nu }{2^{\mu +1}\pi ^{\frac{3}{2}}}{}_{2}\Psi _{0}\left[ 
\begin{array}{c}
\left( \frac{\nu -\mu +1}{2},\frac{1}{2}\right) ,\left( \frac{-\mu -\nu }{2},%
\frac{1}{2}\right) \\ 
----%
\end{array}%
;-t\right]  \nonumber \\
&&-\frac{\sin \pi \mu }{2^{\mu +1}\pi ^{\frac{3}{2}}}{}_{2}\Psi _{0}\left[ 
\begin{array}{c}
\left( \frac{\nu -\mu +1}{2},\frac{1}{2}\right) ,\left( \frac{-\mu -\nu }{2},%
\frac{1}{2}\right) \\ 
----%
\end{array}%
;t\right] \text{.}  \nonumber
\end{eqnarray}
\end{corollary}

Note that even and odd parts of (\ref{New0}) and (\ref{New00}) can be
written in terms of Gauss hypergeometric functions and then Ferrers
functions are expressed in the form similar to that known for associated
Legendre functions.

By making in (\ref{gen1aa}) the$\ $changes $x=\cos \beta $ and $s=\left(
\sin \alpha \right) ^{-1}\sin \left( \alpha -\beta \right) $, we arrive at

\begin{corollary}
If $\beta \in \left( -\pi ,\pi \right) $ and $\alpha =\alpha _{1}+i\alpha
_{2}\in \mathbb{C}$, $\alpha /\pi $ $\notin \mathbb{Z}$, then%
\begin{equation}
\widehat{P}_{\nu }^{-\mu }\left( \cos \alpha \right) =\frac{1}{\sin ^{\nu
}\beta }\sum_{n=0}^{\infty }\frac{\left( \mu -\nu \right) _{n}\sin
^{n}\left( \alpha -\beta \right) }{n!\sin ^{n-\nu }\alpha }P_{n-\nu
-1}^{-\mu }\left( \cos \beta \right) \text{.}  \label{P-cos}
\end{equation}%
When $\nu -\mu \notin \mathbb{N}_{0}$, there are the additional
restrictions: $\sin \left( 2\alpha _{1}-\beta \right) <0$ and (\ref{P-cos})
also holds as $\sin \left( 2\alpha _{1}-\beta \right) =0$ if either $\func{Re%
}\nu >1/2$ or $-1/2<\func{Re}\nu \leq 1/2$ while $\arg \left( \sin \left(
\alpha -\beta \right) /\sin \alpha \right) \neq \pm \beta $.
\end{corollary}

The following formulas originate from (\ref{gen2a}) on changing variables.

\begin{corollary}
Let $\ x,y\in \mathbb{R}$. If $x^{2}+y^{2}<2$, then 
\begin{equation}
P_{\nu }^{-\mu }\left( xy\right) =\frac{4^{-\mu }}{\left(
1-x^{2}y^{2}\right) ^{\frac{\mu }{2}}}\sum_{n=0}^{\infty }\frac{P_{\nu
}^{n-\mu }\left( \frac{\left( x+y\right) ^{2}}{4}\right) \left( x-y\right)
^{2n}}{\left( 16-\left( x+y\right) ^{4}\right) ^{\frac{n-\mu }{2}}n!}\text{.}
\label{uuu1}
\end{equation}%
If $-1<x,y<1$, then%
\begin{equation}
P_{\nu }^{-\mu }\left( x\right) =\frac{2^{-\mu }}{\left( 1-x^{2}\right) ^{%
\frac{\mu }{2}}}\sum_{n=0}^{\infty }\frac{P_{\nu }^{n-\mu }\left( \frac{x+y}{%
2}\right) \left( y-x\right) ^{n}}{\left( 4-\left( x+y\right) ^{2}\right) ^{%
\frac{n-\mu }{2}}n!}\text{.}  \label{uuu2}
\end{equation}%
If $0<\theta +\varphi <\pi $, then%
\begin{eqnarray}
\frac{P_{\nu }^{-\mu }\left( \cos \left( \theta +\varphi \right) \right) }{%
\sin ^{-\mu }\left( \theta +\varphi \right) } &=&\sum_{n=0}^{\infty }\frac{%
\sin ^{n}\theta \sin ^{n}\varphi }{n!}\frac{P_{\nu }^{n-\mu }\left( \cos
\theta \cos \varphi \right) }{\left( 1-\cos ^{2}\theta \cos ^{2}\varphi
\right) ^{\frac{n-\mu }{2}}}\text{,}  \label{uuu3} \\
\frac{2^{\mu }P_{\nu }^{-\mu }\left( \cos \theta \cos \varphi \right) }{%
\left( 1-\cos ^{2}\theta \cos ^{2}\varphi \right) ^{-\frac{\mu }{2}}}
&=&\sum_{n=0}^{\infty }\frac{\left( -1\right) ^{n}\cos ^{n}\left( \theta
-\varphi \right) P_{\nu }^{n-\mu }\left( \frac{\cos \left( \theta +\varphi
\right) }{2}\right) }{\left[ 4-\cos ^{2}\left( \theta +\varphi \right) %
\right] ^{\frac{n-\mu }{2}}n!}\text{.}  \label{uuu4}
\end{eqnarray}
\end{corollary}

A new generating function for Ferrers functions $P_{k}^{\mu -k}\left(
x\right) $ arises as a partial case of the next result.

\begin{theorem}
Let $z>-1$, $\lambda ,t\in \mathbb{C}$. If $\ \left\vert t\right\vert
<z^{-2}-1$ as $-1<$ $z<-2^{-1/2}$ while $-\lambda \notin \mathbb{N}$, $%
\left\vert t\right\vert <1$ as $z\geq -2^{-1/2}$ or $-\lambda \in \mathbb{N}$%
, then 
\begin{equation}
f_{\lambda }\left( t,z\right) =\left( 1+z\sqrt{1-t}\right) ^{-\lambda
}=\sum_{n=0}^{\infty }\frac{\mathfrak{A}_{n}^{\left( \lambda \right) }\left(
z\right) }{2^{n}n!}t^{n}\text{,}  \label{f}
\end{equation}%
where 
\begin{equation}
\mathfrak{A}_{n}^{\left( \lambda \right) }\left( z\right) =\frac{\left(
\lambda \right) _{n}z^{n}}{\left( z+1\right) ^{\lambda +n}}F\left(
1-n,n;1-n-\lambda ;\frac{1+z}{2z}\right) \text{,}  \label{Poly1a}
\end{equation}%
and also if $-\lambda \in \mathbb{N}_{0}$ or $\left\vert z\right\vert <1$, 
\begin{equation}
\mathfrak{A}_{n}^{\left( \lambda \right) }\left( z\right)
=2^{n}\sum_{j=0}^{\infty }\frac{\left( -1\right) ^{j+n}\left( \lambda
\right) _{j}\Gamma \left( \frac{j}{2}+1\right) }{\Gamma \left( \frac{j}{2}%
-n+1\right) }\frac{z^{j}}{j!}\text{.}  \label{U}
\end{equation}%
The functions $\mathfrak{A}_{n}^{\left( \lambda \right) }\left( z\right) $
are analytic functions of the parameter $\lambda $ obeying the recurrence
relations%
\begin{eqnarray}
\mathfrak{A}_{n+1}^{\left( \lambda \right) }\left( z\right) &=&2n\mathfrak{A}%
_{n}^{\left( \lambda \right) }-z\frac{d\mathfrak{A}_{n}^{\left( \lambda
\right) }}{dz}\text{,}  \label{rec-der} \\
\mathfrak{A}_{0}^{\left( \lambda \right) } &=&\frac{1}{\left( 1+z\right)
^{\lambda }}\text{, }\mathfrak{A}_{1}^{\left( \lambda \right) }=\frac{%
\lambda z}{\left( 1+z\right) ^{\lambda +1}}\text{,}  \nonumber
\end{eqnarray}%
and%
\begin{equation}
\left( z^{2}-1\right) \mathfrak{A}_{n+2}^{\left( \lambda \right) }\left(
z\right) +\alpha \left( z,\lambda ,n\right) \mathfrak{A}_{n+1}^{\left(
\lambda \right) }\left( z\right) +\beta \left( \lambda ,n\right) z^{2}%
\mathfrak{A}_{n}^{\left( \lambda \right) }\left( z\right) =0\text{,}
\label{rec-U}
\end{equation}%
where%
\begin{eqnarray*}
\text{ }\alpha \left( z,\lambda ,n\right) &=&2n+1-\left( 4n+2\lambda
+3\right) z^{2}\text{, \ } \\
\beta \left( \lambda ,n\right) &=&4n^{2}+2n\left( 2\lambda +1\right)
+\lambda \left( \lambda +1\right) \text{. }
\end{eqnarray*}
\end{theorem}

\begin{proof}
To obtain the explicit expression (\ref{Poly1a}), we note that the function $%
y=1+z\sqrt{1-t}$ is a solution of the equation $y=a-tz^{2}\left(
y+z-1\right) ^{-1}$ as $a=z+1$. \ Then Lagrange's expansion in powers of $t$
for the function $y^{-\lambda }=\left( 1+z\sqrt{1-t}\right) ^{-\lambda }$%
manifests that as $n\geq 1$ \ and $z>1$,%
\begin{eqnarray*}
\mathfrak{A}_{n}^{\left( \lambda \right) }\left( z\right) &=&2^{n}\left[ 
\frac{d^{n-1}}{da^{n-1}}\left( -\frac{\left( -z^{2}\right) ^{n}\lambda
a^{-\lambda -1}}{\left( a+z-1\right) ^{n}}\right) \right] _{a=z+1} \\
&=&\lambda \left( 2z^{2}\right) ^{n}\left[ \sum_{j=0}^{n-1}\frac{\left(
n-1\right) !\left( \lambda +1\right) _{j}\left( n\right) _{n-1-j}a^{-\lambda
-1-j}}{j!\left( n-j-1\right) !\left( a+z-1\right) ^{2n-j-1}}\right] _{a=z+1}
\\
&=&\frac{\left( z+1\right) ^{-\lambda }}{\Gamma \left( \lambda \right) }%
\sum_{j=0}^{n-1}\frac{\Gamma \left( j+\lambda +1\right) \Gamma \left(
2n-1-j\right) \left( \frac{z}{z+1}\right) ^{j+1}}{j!\left( n-j-1\right)
!2^{n-j-1}}\text{.}
\end{eqnarray*}%
On making the change $j=n-1-k$ and employing Euler's reflection formula for
gamma functions the above sum becomes%
\[
\mathfrak{A}_{n}^{\left( \lambda \right) }\left( z\right) =\frac{z^{n}\Gamma
\left( 1-\lambda \right) }{\left( z+1\right) ^{\lambda +n}}\sum_{k=0}^{n-1}%
\frac{\left( -1\right) ^{n-k}\Gamma \left( n+k\right) \left( \frac{z+1}{2z}%
\right) ^{k}}{\Gamma \left( k+1-n-\lambda \right) \left( n-k-1\right) !k!}%
\text{.} 
\]%
Then, by taking into account that $\left( -m\right) _{k}=\left( -1\right)
^{k}m!/\left( m-k\right) !$ as $m\geq k$, $m,k\in \mathbb{N}_{0}$, 
\begin{eqnarray*}
\mathfrak{A}_{n}^{\left( \lambda \right) }\left( z\right) &=&\frac{%
z^{n}\left( -1\right) ^{n}\Gamma \left( 1-\lambda \right) }{\left(
z+1\right) ^{\lambda +n}\Gamma \left( 1-n-\lambda \right) }\sum_{k=0}^{n-1}%
\frac{\left( n\right) _{k}\left( 1-n\right) _{k}\left( \frac{z+1}{2z}\right)
^{k}}{\left( 1-n-\lambda \right) _{k}k!}\text{,} \\
&=&\frac{z^{n}\Gamma \left( n+\lambda \right) }{\left( z+1\right) ^{\lambda
+n}\Gamma \left( \lambda \right) }\sum_{k=0}^{n-1}\frac{\left( n\right)
_{k}\left( 1-n\right) _{k}\left( \frac{z+1}{2z}\right) ^{k}}{\left(
1-n-\lambda \right) _{k}k!}\text{,}
\end{eqnarray*}%
and we arrive at (\ref{Poly1a}), which holds for $n=0$ also. The conditions
imposed on the radius of convergence of the power series (\ref{f}) arise
from the locations of the singular points of $f_{\lambda }\left( t,z\right) $%
.

Expanding $\left( 1+zu\right) ^{-\lambda }$ with $u=\sqrt{1-t}$, into a
binomial series and substituting a binomial series in powers of $t$ for $%
u^{n}$, on interchanging the order of summation one readily obtains the
expansions (\ref{f}) with $\mathfrak{A}_{n}^{\left( \lambda \right) }\left(
z\right) $ given by the series (\ref{U}) which is absolutely convergent as $%
-1<z<1$ or truncated as $-\lambda \in \mathbb{N}_{0}$.

To derive recurrence relations, note that 
\begin{eqnarray}
\frac{\partial f_{\lambda }}{\partial t} &=&\frac{\lambda z}{2\sqrt{1-t}}%
f_{\lambda +1}\text{,}  \nonumber \\
\text{ }z\frac{\partial f_{\lambda }}{\partial z} &=&-\lambda z\sqrt{1-t}%
f_{\lambda +1}=-2\left( 1-t\right) \frac{\partial f_{\lambda }}{\partial t}%
\text{,}  \label{f2} \\
f_{\lambda +1} &=&f_{\lambda }-z\sqrt{1-t}f_{\lambda +1}=f_{\lambda }-\frac{2%
}{\lambda }\left( 1-t\right) \frac{\partial f_{\lambda }}{\partial t}\text{,}
\nonumber
\end{eqnarray}%
and hence%
\begin{eqnarray}
f_{\lambda } &=&\left( 1+z^{2}\left( 1-t\right) +2z\sqrt{1-t}\right)
f_{\lambda +2}=2f_{\lambda +1}+\left( z^{2}\left( 1-t\right) -1\right)
f_{\lambda +2}  \nonumber \\
&=&2f_{\lambda }-\frac{4}{\lambda }\left( 1-t\right) \frac{\partial
f_{\lambda }}{\partial t}+\left( z^{2}\left( 1-t\right) -1\right) f_{\lambda
+2}\text{.}  \label{f4}
\end{eqnarray}%
Note that%
\begin{eqnarray*}
\frac{\lambda \left( \lambda +1\right) z^{2}}{4}f_{\lambda +2} &=&\sqrt{1-t}%
\frac{\partial }{\partial t}\sqrt{1-t}\frac{\partial f_{\lambda }}{\partial t%
} \\
&=&\left( 1-t\right) \frac{\partial ^{2}f_{\lambda }}{\partial t^{2}}-\frac{1%
}{2}\frac{\partial f_{\lambda }}{\partial t}\text{.}
\end{eqnarray*}%
Then, excluding $f_{\lambda +2}$ from (\ref{f4}) yields the differential
equation 
\begin{equation}
\frac{4\left( z^{2}\left( 1-t\right) -1\right) }{\lambda \left( \lambda
+1\right) z^{2}}\left[ \left( 1-t\right) \frac{\partial ^{2}f_{\lambda }}{%
\partial t^{2}}-\frac{1}{2}\frac{\partial f_{\lambda }}{\partial t}\right] -%
\frac{4}{\lambda }\left( 1-t\right) \frac{\partial f_{\lambda }}{\partial t}%
+f_{\lambda }=0\text{.}  \label{f5}
\end{equation}%
Now, (\ref{rec-U}) is obtained by inserting the series (\ref{f}) into (\ref%
{f5}) and equating coefficients of like powers. In the same way, inserting
the series (\ref{f}) into (\ref{f2}) leads to (\ref{rec-der}). The proof is
completed.
\end{proof}

The series representation (\ref{U}) gives rise to

\begin{corollary}
If $\ r\in \mathbb{N}_{0}$, then the functions $\mathfrak{A}_{n}^{\left(
-r\right) }\left( z\right) $ are polynomials%
\begin{equation}
\mathfrak{A}_{n}^{\left( -r\right) }\left( z\right) =\left( -1\right)
^{n}2^{n}r!\sum_{j=0}^{r}\frac{\Gamma \left( \frac{j}{2}+1\right) }{\left(
r-j\right) !\Gamma \left( \frac{j}{2}-n+1\right) }\frac{z^{j}}{j!}\text{.}
\label{Ur}
\end{equation}
\end{corollary}

\bigskip If $z>1$, then the functions $\mathfrak{A}_{n}^{\left( \lambda
\right) }\left( z\right) $ by virtue of (\ref{1}) and (\ref{Jac1}) are
expressed in terms of Ferrers functions 
\[
\mathfrak{A}_{n}^{\left( \lambda \right) }\left( z\right) =\frac{\left(
-z\right) ^{n}\Gamma \left( 1-\lambda \right) }{\left( z^{2}-1\right) ^{%
\frac{\lambda +n}{2}}}P_{n-1}^{n+\lambda }\left( -\frac{1}{z}\right) =\frac{%
\lambda z^{n}\Gamma \left( 2n+\lambda \right) }{\left( z^{2}-1\right) ^{%
\frac{\lambda +n}{2}}}P_{n-1}^{-n-\lambda }\left( \frac{1}{z}\right) \text{.}
\]%
Making the changes $z=\sec \theta $, $t=s\sin \theta $, $\lambda =-\mu -1$
and $n=k+1$\ in (\ref{f}), we obtain by differentiation with respect to the
parameter $s$

\begin{corollary}
The generating function 
\[
\sum_{k=0}^{\infty }\Gamma \left( 2k+1-\mu \right) P_{k}^{\mu -k}\left( \cos
\theta \right) \frac{s^{k}}{2^{k}k!}=\frac{\left( \cos \theta +\sqrt{1-s\sin
\theta }\right) ^{\mu }}{\sqrt{1-s\sin \theta }\sin ^{\mu }\theta } 
\]
is valid as $\mu ,s\in \mathbb{C}$, $\theta \in \left( 0,\pi /2\right) $,
and $\left\vert s\right\vert \sin \theta <1$.
\end{corollary}

\section{Addition theorems}

In this section, analogs of addition theorems expressing $P_{\nu }^{-\mu
}\left( \tanh \left( \alpha +\beta \right) \right) $ as series of products
of the form $\left( 1+\tanh \beta \tanh \alpha \right) ^{-\nu }V_{n}\left(
\beta \right) \widetilde{V}_{n}\left( \alpha \right) P_{\nu +bn}^{-\mu
-an}\left( \tanh \alpha \right) $ will be obtained. Such theorems are
interesting because $P_{\nu }^{-\mu }\left( \tanh \alpha \right) $ arise in
various problems of mathematical physics described by a Schr\"{o}dinger
operator with a regular P\"{o}schl-Teller potential on the line. The
simplest addition theorem expressing $P_{\nu }^{-\mu }\left( x+y\right) $ in
terms of $P_{\nu }^{n-\mu }\left( x\right) $ is given by (\ref{gen2a}) as $%
x,y\in \mathbb{R}$\ and $\left\vert x\right\vert +\left\vert y\right\vert <1$%
. \ If $x=\left( 1+\tanh \alpha \tanh \beta \right) ^{-1}\tanh \alpha $ and $%
y=\left( 1+\tanh \alpha \tanh \beta \right) ^{-1}\tanh \beta $, then $P_{\nu
}^{-\mu }\left( x+y\right) =P_{\nu }^{-\mu }\left( \tanh \left( \alpha
+\beta \right) \right) $.

We commence with

\begin{theorem}
Let $\gamma ,\mu \in \mathbb{C}$,\ Then the addition formula%
\begin{equation}
\frac{P_{\gamma -1}^{-\mu }\left( \tanh \left( \alpha +\beta \right) \right) 
}{\left( 1+\tanh \beta \tanh \alpha \right) ^{\gamma }}=\sum_{n=0}^{\infty }%
\frac{\left( \gamma \right) _{n}}{n!}\frac{\mathfrak{A}_{n}^{\left( \gamma
+\mu \right) }\left( \tanh \beta \right) }{\cosh ^{\mu }\left( \beta \right) 
}\frac{P_{\gamma +n-1}^{-\mu -n}\left( \tanh \alpha \right) }{\cosh
^{n}\alpha }  \label{sh-theor}
\end{equation}
is valid as $-\gamma -\mu \in \mathbb{N}_{0}$ or $\alpha >0$ and $\cosh
\alpha >-\sinh \beta $.
\end{theorem}

\begin{proof}
According to (\ref{shift2}), for $\func{Re}\gamma >0$ and $\mu -\gamma
=\sigma $, $\func{Re}\sigma \geq 0$, 
\[
\frac{P_{\gamma -1}^{-\sigma -\gamma }\left( \tanh \left( \alpha +\beta
\right) \right) }{\cosh ^{\gamma }\left( \alpha +\beta \right) }=\frac{%
2^{1-\gamma }}{\Gamma \left( \gamma \right) \Gamma \left( \sigma +1\right) }%
\int_{\alpha +\beta }^{\infty }\frac{\sinh ^{\sigma }\left( u-\alpha -\beta
\right) }{\cosh ^{2\gamma +\sigma }u}du\text{,} 
\]%
or by making the change $u=t+\beta $, 
\begin{eqnarray*}
\frac{P_{\gamma -1}^{-\sigma -\gamma }\left( \tanh \left( \alpha +\beta
\right) \right) }{\cosh ^{\gamma }\left( \alpha +\beta \right) } &=&\mathcal{%
W}\left( \beta \right) \int_{\alpha }^{\infty }\frac{\sinh ^{\sigma }\left(
t-\alpha \right) }{\cosh ^{2\gamma +\sigma }t}f_{2\gamma +\sigma }\left( 
\frac{1}{\cosh ^{2}t},\tanh \beta \right) dt\text{,} \\
\mathcal{W}\left( \beta \right) &=&\frac{2^{1-\gamma }}{\Gamma \left( \gamma
\right) \Gamma \left( \sigma +1\right) \cosh ^{2\gamma +\sigma }\beta }\text{%
.}
\end{eqnarray*}%
In the above integral, the function $f_{2\gamma +\sigma }\left( \cosh
^{-2}t,\tanh \beta \right) $ is defined by (\ref{f}) and can be written in
the form of power series in powers of $\cosh ^{-2}t$. Now, (\ref{sh-theor})
is obtained, in this special case, by term-by-term integration of the series
(by virtue of the uniform convergence) and making use of (\ref{shift2}). By
using the estimate (\ref{est})$\ $and the asymptotic formula (\ref{As-gam}),
one can see that absolute values of terms of the series in (\ref{sh-theor})
are less than 
\[
A\left( \mu ,\gamma ,\alpha \right) \sum_{n=0}^{\infty }\frac{\left\vert 
\mathfrak{A}_{n}^{\left( \lambda \right) }\left( \tanh \beta \right)
\right\vert }{2^{n}n!}\frac{n^{\gamma -\mu -1}}{\cosh ^{2n}\alpha }\text{.} 
\]%
Then, we infer by employing the preceding theorem that the series in (\ref%
{sh-theor}) converges absolutely and uniformly with respect to $\left( \mu
,\gamma \right) $ belonging to any bounded region of $\mathbb{C}^{2}$, and
therefore it is an analytic function of parameters $\mu $ and $\gamma $. The
proof of the theorem is completed by analytic continuation.
\end{proof}

\begin{theorem}
The addition formula%
\begin{eqnarray}
\frac{P_{\nu }^{-\mu }\left( \tanh \left( \alpha +\beta \right) \right) }{%
\left( 1+\tanh \beta \tanh \alpha \right) ^{-\nu }} &=&\sum_{n=0}^{\infty }%
\frac{\left( \mu -\nu \right) _{n}\left( \nu \right) _{n}\mathcal{F}%
_{n}\left( \beta \right) }{2^{\nu }\left( -1\right) ^{n}n!e^{\alpha n}}%
P_{\nu }^{-\mu -n}\left( \tanh \alpha \right) \text{,}  \label{adi1} \\
\text{ }\mathcal{F}_{n}\left( \beta \right) &=&\frac{e^{\beta \left( \nu
-\mu \right) }}{\cosh ^{\nu }\beta }F\left( -n,-\nu ;1-\nu -n;e^{-2\beta
}\right) \text{, }\alpha >0\text{, }  \nonumber
\end{eqnarray}%
is valid as $\alpha +\beta >0$ or as $\beta \in \mathbb{R}$ and $\nu \in 
\mathbb{N}_{0}$.
\end{theorem}

\begin{proof}
According to (\ref{shift}), as $\mu >\nu >0$, 
\[
\frac{P_{\nu }^{-\mu }\left( \tanh \left( \alpha +\beta \right) \right) }{%
\cosh ^{-\nu }\left( \alpha +\beta \right) }=\frac{2^{1+\nu }e^{\left(
\alpha +\beta \right) \mu }}{\Gamma \left( \mu -\nu \right) \Gamma \left(
1+\nu \right) }\int_{\alpha +\beta }^{\infty }\frac{e^{-2u\mu }\cosh ^{\nu
}udu}{\sinh ^{-\nu }\left( u-\alpha -\beta \right) }\text{,} 
\]%
or on making the change $u=s+\beta $ and using the identity $\cosh ^{\nu
}\left( s+\beta \right) =g_{\nu }\left( e^{-2s},e^{-2\beta }\right) e^{\beta
\nu }\cosh ^{\nu }s$ with $g_{\nu }\left( t,z\right) =\left( 1+zt\right)
^{\nu }\left( 1+t\right) ^{-\nu }$, 
\begin{eqnarray}
P_{\nu }^{-\mu }\left( \tanh \left( \alpha +\beta \right) \right) &=&%
\mathcal{P}\left( \alpha ,\beta \right) \int_{\alpha }^{\infty }\frac{%
e^{-2s\mu }\text{ }g_{\nu }\left( e^{-2s},e^{-2\beta }\right) ds}{\left[
\sinh \left( s-\alpha \right) \cosh s\right] ^{-\nu }}\text{,}  \label{121}
\\
\mathcal{P}\left( \alpha ,\beta \right) &=&\frac{2^{1+\nu }e^{\alpha \mu
}e^{\left( \nu -\mu \right) \beta }\cosh ^{-\nu }\left( \alpha +\beta
\right) }{\Gamma \left( \mu -\nu \right) \Gamma \left( 1+\nu \right) }\text{.%
}  \nonumber
\end{eqnarray}%
Expand $g_{\nu }\left( t,z\right) $ into a series in powers of $t$. On
multiplying binomial series for $\left( 1+t\right) ^{-\nu }$ and $\left(
1+zt\right) ^{\nu }$ we have as $\left\vert zt\right\vert ,\left\vert
t\right\vert <1$ or as $\left\vert t\right\vert <1$ and $\nu \in \mathbb{N}%
_{0}$, $\qquad $ 
\begin{eqnarray*}
g_{\nu }\left( t,z\right) &=&\left( \sum_{k=0}^{\infty }\frac{\Gamma \left(
1-\nu \right) t^{k}}{\Gamma \left( 1-\nu -k\right) k!}\right) \left(
\sum_{k=0}^{\infty }\frac{\left( -1\right) ^{k}\Gamma \left( k-\nu \right) }{%
\Gamma \left( -\nu \right) k!}t^{k}z^{k}\right) \\
&=&\sum_{n=0}^{\infty }\left( \sum_{k=0}^{n}\frac{\left( -1\right)
^{k}\Gamma \left( 1-\nu \right) \Gamma \left( k-\nu \right) z^{k}}{\Gamma
\left( -\nu \right) \Gamma \left( 1-\nu -n+k\right) \left( n-k\right) !k!}%
\right) t^{n}\text{.}
\end{eqnarray*}%
Now, $g_{\nu }\left( t,z\right) $ is seen to be a generating function for \
hypergeometric polynomials 
\begin{eqnarray}
g_{\nu }\left( t,z\right) &=&\sum_{n=0}^{\infty }\frac{\Gamma \left( 1-\nu
\right) }{n!\Gamma \left( 1-\nu -n\right) }\left( \sum_{k=0}^{n}\frac{\left(
-n\right) _{k}\left( -\nu \right) _{k}}{\left( 1-\nu -n\right) _{k}}\frac{%
z^{k}}{k!}\right) t^{n}  \nonumber \\
&=&\sum_{n=0}^{\infty }\frac{\left( -1\right) ^{n}\left( \nu \right) _{n}}{n!%
}F\left( -n,-\nu ;1-\nu -n;z\right) t^{n}\text{,}  \label{122}
\end{eqnarray}%
and therefore $g_{\nu }\left( e^{-2s},e^{-2\beta }\right) $ in (\ref{121})
can be written as an absolutely and uniformly convergent power series in
powers of $e^{-2s}$ on any interval $s\geq \alpha _{0}>0$. Now, (\ref{adi1})
follows for $\mu >\nu >0$ as the\ result of term-by-term integration of the
series (by virtue of the uniform convergence) and making use of (\ref{shift}%
).\ Note that the\ series in (\ref{adi1}) converges absolutely and uniformly
with respect to parameters on any bounded region of $\mu $ and $\nu $ for $%
\alpha \geq \alpha _{0}>0$ due to the estimate%
\begin{eqnarray*}
&&\left\vert \sum_{n=0}^{\infty }\frac{\left( \mu -\nu \right) _{n}\left(
\nu \right) _{n}\mathcal{F}_{n}\left( \beta \right) }{\left( -1\right)
^{n}n!e^{\alpha n}}P_{\nu }^{-\mu -n}\left( \tanh \alpha \right) \right\vert
\\
&\leq &\mathfrak{D}_{0}\sum_{n=0}^{\infty }\left\vert \frac{\left( \nu
\right) _{n}\mathcal{F}_{n}\left( \beta \right) }{\left( -1\right) ^{n}n!}%
\right\vert e^{-\alpha n}\frac{e^{-\alpha n}}{\left\vert n^{\nu
+1}\right\vert }\leq \mathfrak{D}\sum_{n=0}^{\infty }\frac{\left\vert \left(
\nu \right) _{n}\mathcal{F}_{n}\left( \beta \right) \right\vert }{n!}%
e^{-\alpha n}
\end{eqnarray*}%
involving absolute convergence of the power series (\ref{122}), and
therefore it is an analytic function of parameters $\mu $ and $\gamma $.
Then analytic continuation brings in the proof of \ the theorem.
\end{proof}

\begin{theorem}
\bigskip The addition formula%
\begin{eqnarray}
\frac{P_{\nu }^{-\mu }\left( \tanh \left( \alpha +\beta \right) \right) }{%
\left( 1+\tanh \beta \tanh \alpha \right) ^{-\nu }} &=&\sum_{n=0}^{\infty }%
\frac{\left( -\nu \right) _{n}\left( 2\mu \right) _{n}F_{n}\left( \beta
\right) P_{\mu +\nu }^{-n-\mu }\left( \tanh \alpha \right) }{\left(
-1\right) ^{n}n!e^{\left( \mu -\nu \right) \beta }e^{\left( \alpha +\mu
\right) n}\cosh ^{-\nu -\mu }\alpha }\text{,}  \label{adi2} \\
\text{ }F_{n}\left( \beta \right) &=&F\left( -n,\mu -\nu ;2\mu ;1-e^{-2\beta
}\right) \text{,}  \nonumber
\end{eqnarray}%
is valid for $\mu ,\nu \in \mathbb{C}$ if $\alpha >0$ and $\alpha +\beta
\geq 0$.
\end{theorem}

\begin{proof}
Initially, according to (\ref{shift2}), as $\func{Re}\gamma >0$ and $\func{Re%
}\sigma \geq 0$,%
\begin{eqnarray*}
\frac{P_{\gamma -1}^{-\sigma -\gamma }\left( \tanh \left( \alpha +\beta
\right) \right) }{\cosh ^{\gamma }\left( \alpha +\beta \right) } &=&\frac{%
2^{1-\gamma }}{\Gamma \left( \gamma \right) \Gamma \left( \sigma +1\right) }%
\int_{\alpha +\beta }^{\infty }\frac{\sinh ^{\sigma }\left( u-\alpha -\beta
\right) }{\cosh ^{2\gamma +\sigma }u}du \\
&=&\frac{2^{\gamma +2\sigma +1}}{\Gamma \left( \gamma \right) \Gamma \left(
\sigma +1\right) }\int_{\alpha }^{\infty }\frac{e^{-2\left( \gamma +\sigma
\right) u}q_{\sigma ,\gamma }\left( e^{-2u},e^{-2\beta }\right) du}{%
e^{\left( 2\gamma +\sigma \right) \beta }\left( \sinh \left( u-\alpha
\right) \cosh u\right) ^{-\sigma }}\text{,}
\end{eqnarray*}%
where the function $q_{\sigma ,\gamma }\left( t,z\right) =\left( 1+t\right)
^{-\sigma }\left( 1+zt\right) ^{-2\gamma -\sigma }$is a generating function
for \ hypergeometric polynomials \cite{Erdelyi} 
\begin{eqnarray}
q_{\sigma ,\gamma }\left( t,z\right) &=&\sum_{n=0}^{\infty }\frac{\left(
2\sigma +2\gamma \right) _{n}}{\left( -1\right) ^{n}n!}F\left( -n,2\gamma
+\sigma ;2\sigma +2\gamma ;1-z\right) t^{n}\text{,}  \label{hyp333} \\
\text{ }\left\vert tz\right\vert &<&1\text{, }\left\vert t\right\vert <1%
\text{.}  \nonumber
\end{eqnarray}%
By inserting (\ref{hyp333}) and integrating term-by-term with exploiting (%
\ref{shift}), one obtains (\ref{adi2}) in the considered initial case on
denoting $\mu =\sigma +\gamma \ $and $\nu =-\gamma $. Now, the proof of the
theorem for any $\mu ,\nu \in \mathbb{C}$ is accomplished by analytic
continuation based on the uniform convergence of the series (\ref{adi2}).
The mentioned uniform convergence can be readily shown by using the absolute
convergence of (\ref{hyp333}) and the asymptotic formula (\ref{est}).
\end{proof}

\section{Gegenbauer polynomials}

By setting $\nu =n+\tau -\frac{1}{2}$\ in (\ref{ss}) and employing (\ref%
{Geg0}), we obtain the classical Rodrigues formula for Gegenbauer
polynomials. The differential relations (\ref{Dif2}), (\ref{Dif2a}) (\ref%
{Dif3a}) and (\ref{Dif5}),\ together with (\ref{Geg0}), are sources of
various other Rodriges-like formulas. In particular, one gets by setting $%
\mu =\tau -\frac{1}{2}$ and $\nu =-\tau -\frac{1}{2}$ and using (\ref{F}):
from (\ref{Dif5}) as $-\infty <t<\infty $,%
\[
C_{n}^{\left( \tau \right) }\left( \frac{t}{\sqrt{1+t^{2}}}\right) =\frac{%
\left( -1\right) ^{n}}{n!}\left( 1+t^{2}\right) ^{\frac{n+2\tau }{2}}\frac{%
d^{n}}{dt^{n}}\left( 1+t^{2}\right) ^{-\tau }\text{;} 
\]%
from (\ref{Dif1-1}) as $\sinh \alpha =1/t-z$, $z\in \mathbb{R}$, and $%
0<t\leq \infty $,%
\[
C_{n}^{\left( \tau \right) }\left( \frac{1-tz}{\sqrt{t^{2}+\left(
1-tz\right) ^{2}}}\right) =\frac{\left( t^{2}+\left( 1-tz\right) ^{2}\right)
^{\frac{n+2\tau }{2}}}{n!t^{^{2\tau -1}}}\frac{d^{n}}{dt^{n}}\frac{%
t^{n+2\tau -1}}{\left( t^{2}+\left( 1-tz\right) ^{2}\right) ^{\tau }}\text{;}
\]%
and from (\ref{Dif2a}) as $z=i$ and $x=\cos \theta $, $0<\theta <\pi $,%
\[
C_{n}^{\left( \tau \right) }\left( \cos \theta \right) =\frac{e^{-i\left(
n+1\right) \theta }}{n!\sin ^{2\tau -1}\theta }\left( e^{2i\theta }\frac{d}{%
d\theta }\right) ^{n}e^{-i\left( n-1\right) \theta }\sin ^{n+2\tau -1}\theta 
\text{.} 
\]%
The first of the above Rodriges-like formulas is akin to the Tricomi formula
derived in a more complicated manner \cite{Erdelyi}.

\ Writing in (\ref{ss}) $P_{-1-\nu }^{n-\nu }\left( x\right) $ instead of $%
P_{\nu }^{n-\nu }\left( x\right) $ and setting $\nu =-\frac{1}{2}-m$, $n-\nu
=k+\frac{1}{2}$ yield as $k,m\in \mathbb{N}_{0}$ and $k\geq m$, 
\[
\frac{C_{m+k}^{\left( -k\right) }\left( x\right) }{\left( 1-x^{2}\right) ^{k+%
\frac{1}{2}}}=\frac{\left( -1\right) ^{k+m}2^{m+k+1}k!\left( 2m\right) !}{%
\left( k+m\right) !\left( k-m\right) !m!}\frac{d^{k-m}}{dx^{k-m}}\left(
1-x^{2}\right) ^{-\frac{1}{2}-m}\text{.} 
\]

The well-known explicit representation for Gegenbauer polynomials \cite{Sego}
\[
C_{l}^{\left( \tau \right) }\left( x\right) =\frac{1}{\Gamma \left( \tau
\right) }\sum_{j=0}^{\left[ \frac{l}{2}\right] }\frac{\left( -1\right)
^{j}\Gamma \left( l-j+\tau \right) }{j!\left( l-2j\right) !}\left( 2x\right)
^{l-2j}\text{.} 
\]%
follows from (\ref{F}) and (\ref{Geg0}) as $\lambda =\tau -\frac{1}{2}+l$, $%
l\in \mathbb{N}$.

Another simple explicit representation originates from (\ref{New1}), namely 
\[
C_{l}^{\left( \tau \right) }\left( t\right) =\left( 2\tau \right)
_{l}\sum_{j=0}^{\left[ \frac{l}{2}\right] }\frac{t^{l-2j}\left(
1-t^{2}\right) ^{j}}{2^{2j}j!\left( l-2j\right) !\left( \tau +\frac{1}{2}%
\right) _{j}}\text{.} 
\]%
Readers can find a completely different derivation of the above\ formula in 
\cite{Rain}.

A relation connecting different Gegenbauer polynomials can be derived from (%
\ref{P2}) as $\nu =-l-\mu -1/2$, $\tau =l+2\mu $, and $l\in \mathbb{N}$, 
\begin{eqnarray*}
C_{l}^{\left( \mu \right) }\left( x\right) &=&\frac{\sqrt{\pi }\Gamma \left(
l+2\mu \right) \Gamma \left( l+\mu +1\right) }{2^{-l-1}\Gamma \left( \mu
\right) }\sum_{n=0}^{l}\frac{\Gamma ^{-1}\left( \frac{n-l+1}{2}\right)
C_{n}^{\left( l+2\mu \right) }\left( x\right) }{2^{n}\Gamma \left( \frac{%
n+3l+4\mu +2}{2}\right) \left( l-n\right) !} \\
&=&2\frac{\Gamma \left( l+2\mu \right) \Gamma \left( l+\mu +1\right) }{%
\Gamma \left( \mu \right) }\sum_{j=0}^{\left[ l/2\right] }\frac{\left(
-1\right) ^{j}C_{l-2j}^{\left( l+2\mu \right) }\left( x\right) }{j!\Gamma
\left( 2l+2\mu -j+1\right) }\text{.}
\end{eqnarray*}%
It is an elegant special case of the known general relation established by
basing on properties of Jacobi polynomials \cite[Theorem 7.1.4]{Andr}.

Three other connection relations can be established from (\ref{newA}), (\ref%
{new12}) and (\ref{con!5}):%
\[
C_{k}^{\left( 2\mu +\frac{1}{2}\right) }\left( x\right) =\frac{\left( 4\mu
+1\right) _{k}}{\left( 2\mu +1\right) _{k}}\sum_{n=0}^{\left[ \frac{k}{2}%
\right] }\frac{\left( 2\mu \right) _{2n}\left( 1-x^{2}\right)
^{n}C_{k-2n}^{\left( n+\mu +\frac{1}{2}\right) }\left( x\right) }{%
2^{2n}n!\left( 2\mu +1\right) _{n}}\text{,} 
\]%
\[
C_{k}^{\left( \mu +\frac{1}{2}\right) }\left( x\right) =\frac{\Gamma \left(
k+2\mu +1\right) }{2^{-4\mu }\sqrt{\pi }}\sum_{n=0}^{\left[ \frac{k}{2}%
\right] }\frac{4^{n}\left( 2\mu \right) _{n}\Gamma \left( 2n+2\mu +\frac{1}{2%
}\right) C_{k-2n}^{\left( 2n+2\mu +\frac{1}{2}\right) }\left( x\right) }{%
\left( -1\right) ^{n}n!\Gamma \left( 2n+k+4\mu +1\right) \left(
1-x^{2}\right) ^{-n}}\text{,} 
\]%
\[
C_{k}^{\left( \mu +\frac{1}{2}\right) }\left( x\right) =\frac{2\mu +1}{%
k+2\mu +1}\sum_{n=0}^{\left[ \frac{k}{2}\right] }\frac{\left( 2n\right)
!\left( 1-x^{2}\right) ^{n}}{\left( -1\right) ^{n}2^{2n}n!}C_{k-2n}^{\left(
n+\mu +1\right) }\left( x\right) \text{.} 
\]

A relation expressing an Gegenbauer polynomial as a sum of products of
Gegenbauer polynomials follows from (\ref{G0}) as $\nu =l+\mu $, $l\in 
\mathbb{N}_{0}$, and $\mu =\tau -\frac{1}{2}$,

\begin{eqnarray*}
C_{l}^{\left( \tau \right) }\left( x\right) &=&a\left( \tau ,\lambda
,l\right) \sum_{n=0}^{l}b_{n}\left( \tau ,\lambda ,l\right) C_{n}^{\left(
\lambda -n+\frac{1}{2}\right) }\left( x\right) C_{l-n}^{\left( \tau -\lambda
+n\right) }\left( x\right) \text{,} \\
a\left( \tau ,\lambda ,l\right) &=&\frac{\Gamma \left( l+2\tau \right)
\Gamma \left( l+\tau -\lambda +\frac{1}{2}\right) \Gamma \left( 1+\lambda
\right) }{\sqrt{\pi }\Gamma \left( \tau \right) \Gamma \left( l+\tau +\frac{1%
}{2}\right) }\text{,} \\
b_{n}\left( \tau ,\lambda ,l\right) &=&\frac{\Gamma \left( \lambda -n+\frac{1%
}{2}\right) \Gamma \left( n+\tau -\lambda \right) }{\Gamma \left( 2\lambda
-n+1\right) \Gamma \left( n+l+2\tau -2\lambda \right) }\text{,} \\
\text{ }\lambda -\tau -l+\frac{1}{2} &\notin &\mathbb{N}\text{,}-\lambda
\notin \mathbb{N}\text{, }\lambda -\tau \notin \mathbb{N}_{0}\text{, }%
l-\lambda -\frac{1}{2}\notin \mathbb{N}_{0}\text{.}
\end{eqnarray*}

Results obtained in section 7 and section 8 give a number of additional
formulas. Two of them yield finite sums as $\mu =\tau -\frac{1}{2}$ and $\nu
=l+\tau -\frac{1}{2}$, $l\in \mathbb{N}_{0}$. Then, the addition theorem (%
\ref{adi1}) turns into an addition theorem for Gegenbauer polynomials: 
\[
C_{l}^{\left( \tau \right) }\left( \frac{x+y}{1+xy}\right) =\frac{1}{\left(
1+xy\right) ^{l}}\sum_{n=0}^{l}\frac{\left( \tau \right) _{n}\left( l+\tau -%
\frac{1}{2}\right) _{n}}{2^{-n}n!\left( l+2\tau \right) _{n}}\frac{\Omega
_{n}\left( y\right) C_{l-n}^{\left( \tau +n\right) }\left( x\right) }{\left(
1-x\right) ^{-n}}\text{,} 
\]%
where $0<x<1$, $0<x+y<1$, and $\Omega _{n}\left( y\right) $ are polynomials,
namely 
\[
\Omega _{n}\left( y\right) =\left( 1+y\right) ^{l}F\left( -n,\frac{1}{2}%
-\tau -l;\frac{3}{2}-\tau -l-n;\frac{1-y}{1+y}\right) \text{.} 
\]%
A finite sum case of the generating function (\ref{gen1aa}) 
\begin{eqnarray}
C_{l}^{\left( \tau \right) }\left( \frac{x-s}{\sqrt{1-2sx+s^{2}}\ }\right)
&=&\left( 1-2sx+s^{2}\right) ^{^{-\frac{l}{2}}}\sum_{j=0}^{l}\frac{\left(
2\tau \right) _{l}\left( -1\right) ^{l+j}}{\left( 2\tau \right) _{j}\left(
l-j\right) !}C_{j}^{\left( \tau \right) }\left( x\right) s^{l-j}\text{,}
\label{add1} \\
x &\in &\left( -1,1\right) ,s\in \mathbb{C}\text{,}  \nonumber
\end{eqnarray}%
is known \cite{PBM3}. \ 

The above expression involves helpful identities: as $s=2x$,

\[
C_{l}^{\left( \tau \right) }\left( x\right) =\sum_{j=0}^{l}\frac{\left(
2\tau \right) _{l}\left( -1\right) ^{j}}{\left( 2\tau \right) _{j}\left(
l-j\right) !}C_{j}^{\left( \tau \right) }\left( x\right) \left( 2x\right)
^{l-j}\text{;} 
\]%
as $s=\pm 1$, 
\[
C_{l}^{\left( \tau \right) }\left( \sqrt{\frac{1\pm x}{2}}\right) =\frac{%
\Gamma \left( l+2\tau \right) }{2^{\frac{l}{2}}\left( 1\pm x\right) ^{\frac{l%
}{2}}}\sum_{j=0}^{l}\frac{\left( -1\right) ^{\frac{j\mp j}{2}}C_{j}^{\left(
\tau \right) }\left( x\right) }{\Gamma \left( j+2\tau \right) \left(
l-j\right) !}\text{;} 
\]%
as $s=1/x$,%
\[
C_{l}^{\left( \tau \right) }\left( \sqrt{1-x^{2}}\right) =\left(
1-x^{2}\right) ^{^{-\frac{l}{2}}}\sum_{j=0}^{l}\frac{\left( 2\tau \right)
_{l}\left( -1\right) ^{j}x^{j}}{\left( 2\tau \right) _{j}\left( l-j\right) !}%
C_{j}^{\left( \tau \right) }\left( x\right) \text{.} 
\]%
In this case, (\ref{P-cos}) turns into the relation 
\begin{eqnarray*}
C_{l}^{\left( \tau \right) }\left( \cos \alpha \right) &=&\left( \frac{\sin
\alpha }{\sin \beta }\right) ^{l}\sum_{j=0}^{l}\frac{\left( 2\tau \right)
_{l}}{\left( 2\tau \right) _{j}\left( l-j\right) !}\left( \frac{\sin \left(
\beta -\alpha \right) }{\sin \alpha }\right) ^{l-j}C_{j}^{\left( \tau
\right) }\left( \cos \beta \right) \text{,} \\
\beta &\in &\left( 0,\pi \right) \text{, }\alpha \in \mathbb{C}\text{, }%
\alpha /\pi \text{ }\notin \mathbb{Z}\text{,}
\end{eqnarray*}%
which was derived in \cite{Rain} by inventive manipulations in a rather
tedious way.

In conclusion, we note that a number of integral and series representations
for Gegenbauer polynomials, series connecting or containing Gegenbauer
polynomials and integrals of Gegenbauer polynomials can be written by basing
on results obtained in this article. For example, new fractional type
integral operators relating Gegenbauer polynomials and Jacobi polynomials $%
P_{k}^{\left( \mu ,-\mu \right) }\left( x\right) $ follows from (\ref{cc0})
and (\ref{cc1}) by setting $\lambda =\mu $, $k=n+1$, and $\lambda =-\mu $, $%
k=n+1$, respectively: 
\begin{eqnarray*}
\frac{P_{n}^{\left( \mu ,-\mu \right) }\left( \tanh \alpha \right) }{e^{\mu
\alpha }\cosh ^{n+1}\alpha } &=&\frac{\left( n+1\right) \Gamma \left( \mu +%
\frac{1}{2}\right) }{\sqrt{\pi }2^{-\mu }\Gamma \left( 1+\mu \right) }%
\int_{\alpha }^{\infty }\frac{C_{n+1}^{\left( \mu +\frac{1}{2}\right)
}\left( \tanh s\right) ds}{\left( \sinh s-\sinh \alpha \right) ^{-\mu }\cosh
^{n+2\mu +1}s}\text{,} \\
\func{Re}\mu &>&-1\text{, }n\in \mathbb{N}_{0}\text{,}
\end{eqnarray*}

\begin{eqnarray*}
\frac{C_{n}^{\left( \mu +\frac{1}{2}\right) }\left( \tanh \alpha \right) }{%
\cosh ^{n+2\mu +1}\alpha } &=&\frac{2^{-\mu }\sqrt{\pi }\left( n+1\right) }{%
\Gamma \left( 1-\mu \right) \Gamma \left( \mu +\frac{1}{2}\right) }%
\int_{\alpha }^{\infty }\frac{e^{-\mu s}P_{n+1}^{\left( \mu ,-\mu \right)
}\left( \tanh s\right) ds}{\left( \sinh s-\sinh \alpha \right) ^{\mu }\cosh
^{n+1}s}\text{,} \\
-\frac{n+1}{2} &<&\func{Re}\mu <1\text{, }n\in \mathbb{N}_{0}\text{.}
\end{eqnarray*}

\section{Ferrers associated Legendre functions}

Ferrers associated Legendre functions (associated Legendre polynomials) are
defined as $P_{k}^{\pm m}\left( x\right) $ with $k,m\in \mathbb{N}_{0}$ and $%
m\leq k$.

\ As $\nu =k$, $n=k-m$, and $\mu =m$, (\ref{Jac1}) takes the form of the
well-known relations%
\begin{equation}
P_{k}^{\pm m}\left( -x\right) =\left( -1\right) ^{k+m}P_{k}^{\pm m}\left(
x\right) \text{,}  \label{P-m0}
\end{equation}%
\begin{equation}
P_{k}^{-m}\left( x\right) =\left( -1\right) ^{m}\frac{\left( k-m\right) !}{%
\left( k+m\right) !}P_{k}^{m}\left( x\right) \text{.}  \label{P-m}
\end{equation}%
If an order $m$ is fixed and $k-m\in \mathbb{N}_{0}$,

\[
P_{k}^{m}\left( x\right) =\left( -1\right) ^{m}\frac{\left( k+m\right) }{%
\left( k-m\right) !}P_{\left( k-m\right) +m}^{-m}\left( x\right) =\frac{%
\left( 2m\right) !\left( 1-x^{2}\right) ^{\frac{m}{2}}}{\left( -2\right)
^{m}!m!}C_{k-m}^{\left( m+\frac{1}{2}\right) }\left( x\right) \text{,} 
\]%
and therefore it is a complete orthogonal system on $\left( -1,1\right) $
(see also \cite{Ake} on orthogonal properties of Gegenbauer polynomials (and
then associated Legendre functions) on the interior of an ellipse in the
complex plane).\ 

The classical Rodrigues formula for Ferrers associated Legendre functions is
given by (\ref{ss}) where $\nu =k$ and $n=m+k$. Various Rodrigues-like
relations can be\ derived from the differential relations (\ref{Dif2}), (\ref%
{Dif2a}), (\ref{Dif3a}), and (\ref{Dif5}). In particular, by setting $\mu
=m, $ $\nu =-m-1$, and $n=k-m$ into (\ref{Dif5}) and (\ref{Dif2a}) and by
using (\ref{P-m}), one obtains%
\[
P_{k}^{m}\left( \frac{t}{\sqrt{1+t^{2}}}\right) =\frac{2^{-m}\left(
2m\right) !\left( 1+t^{2}\right) ^{\frac{k+1}{2}}}{\left( -1\right)
^{k}m!\left( k-m\right) !}\frac{d^{k-m}}{dt^{k-m}}\left( 1+t^{2}\right) ^{-%
\frac{2m+1}{2}}\text{,} 
\]%
\begin{equation}
P_{k}^{m}\left( \cos \theta \right) =\frac{\left( 2m\right) !e^{-i\left(
k-m+1\right) \theta }}{\left( -2\right) ^{m}m!\left( k-m\right) !\sin
^{m}\theta }\left( e^{2i\theta }\frac{d}{d\theta }\right) ^{k-m}\frac{\sin
^{m+k}\theta }{e^{i\left( k-m-1\right) \theta }}\text{.}  \nonumber
\end{equation}%
In another case, as $\sigma =\nu =k$, $n=k\pm m$, and $z=\pm 1$, (\ref{Dif3a}%
) becomes%
\[
P_{k}^{m}\left( x\right) =\frac{\left( -1\right) ^{k+m}\left( 1-x^{2}\right)
^{\frac{m}{2}}}{2^{k}k!\left( 1\pm x\right) ^{k+m+1}}\left( \left( 1\pm
x\right) ^{2}\frac{d}{dx}\right) ^{k+m}\frac{\left( 1-x^{2}\right) ^{k}}{%
\left( 1\pm x\right) ^{k+m-1}}\text{,} 
\]%
\[
P_{k}^{m}\left( x\right) =\frac{\left( -1\right) ^{k}\left( k+m\right)
!\left( 1-x^{2}\right) ^{-\frac{m}{2}}}{2^{k}k!\left( k-m\right) !\left(
1\pm x\right) ^{k-m+1}}\left( \left( 1\pm x\right) ^{2}\frac{d}{dx}\right)
^{k-m}\frac{\left( 1-x^{2}\right) ^{k}}{\left( 1\pm x\right) ^{k-m-1}}\text{.%
} 
\]

The \ known explicit representations for Ferrers associated Legendre
functions

\begin{eqnarray*}
P_{k}^{m}\left( x\right) &=&\frac{\left( 1-x^{2}\right) ^{\frac{m}{2}}}{2^{k}%
}\sum_{j=0}^{\left[ \frac{k-m}{2}\right] }\frac{\left( -1\right)
^{j+m}\left( 2k-2j\right) !x^{k-m-2j}}{j!\left( k-m-2j\right) !\left(
k-j\right) !}\text{,} \\
P_{k}^{m}\left( x\right) &=&\frac{\left( k+m\right) !\left( 1-x^{2}\right)
^{-\frac{m}{2}}}{2^{k}\left( k-m\right) !}\sum_{j=0}^{\left[ \frac{k+m}{2}%
\right] }\frac{\left( -1\right) ^{j}\left( 2k-2j\right) !x^{k+m-2j}}{%
j!\left( k+m-2j\right) !\left( k-j\right) !}\text{, } \\
P_{k}^{m}\left( x\right) &=&\frac{\left( k+m\right) !}{2^{m}}\sum_{j=0}^{%
\left[ \frac{k-m}{2}\right] }\frac{\left( -1\right) ^{j+m}x^{k-m-2j}\left(
1-x^{2}\right) ^{j+\frac{m}{2}}}{2^{2j}j!\left( k-m-2j\right) !\left(
j+m\right) !}\text{,}
\end{eqnarray*}%
follow from (\ref{F}) and (\ref{New1}). As $x=1$ and $m\geq 1$,\ the above
representations imply the combinatorial identity

\[
\sum_{j=0}^{\left[ \frac{n}{2}\right] }\frac{\left( -1\right) ^{j}\left(
2k-2j\right) !}{j!\left( n-2j\right) !\left( k-j\right) !}=0\text{, }k,n\in 
\mathbb{N}\text{ and }k<n\leq 2k\text{. } 
\]

Connection relations for associated Legendre functions are derived from (\ref%
{qq1}) as $\mu =\pm m$ and from (\ref{qq2}) as $l=2m$ and $\nu =-k-1$. On
employing (\ref{P-m}) we obtain:

\begin{eqnarray*}
\frac{P_{k}^{m}\left( x\right) }{\left( 1+x\right) ^{^{-m}}} &=&\frac{%
2^{m}\left( k+m\right) !k!}{\left( k-m\right) !}\sum_{n=0}^{k}\left( \frac{%
1-x}{1+x}\right) ^{\frac{n}{2}}\frac{\left( -1\right) ^{n}\left( 2m\right)
_{n}P_{k+m}^{m+n}\left( x\right) }{n!\left( k+2m+n\right) !}\text{,} \\
\frac{P_{k}^{m}\left( x\right) }{\left( 1+x\right) ^{m}} &=&\frac{\left(
2m\right) !k!}{2^{m}}\sum_{n=\max \left( 2m-k,0\right) }^{\min \left(
2m,k\right) }\left( \frac{1-x}{1+x}\right) ^{\frac{n}{2}}\frac{\left(
-1\right) ^{n}P_{k-m}^{m-n}\left( x\right) }{\left( k-n\right) !\left(
2m-n\right) !n!}\text{,} \\
\frac{P_{k}^{m}\left( x\right) }{\left( 1+x\right) ^{m}} &=&\frac{\left(
2m\right) !}{2^{m}k!}\sum_{n=0}^{2m}\left( \frac{1-x}{1+x}\right) ^{\frac{n}{%
2}}\frac{\left( k+n\right) !P_{k+m}^{m-n}\left( x\right) }{n!\left(
2m-n\right) !}\text{.}
\end{eqnarray*}%
Note that the above relations can be derived from (\ref{col1}) and (\ref%
{col2}) as well.

Another connection relation can be found from (\ref{newA}) as $\mu =m$ and $%
k=l-m$: 
\[
P_{l+m}^{2m}\left( x\right) =\frac{\left( l+3m\right) !m}{\left( l+m\right) !%
}\sum_{n=0}^{\left[ \frac{l-m}{2}\right] }\frac{\left( m+n-1\right) !\left(
1-x^{2}\right) ^{\frac{n+m}{2}}P_{l-n}^{n+m}\left( x\right) }{\left(
-2\right) ^{n+m}\left( 2m+n\right) !n!}\text{.} 
\]%
A reciprocal formula%
\[
P_{l}^{m}\left( x\right) =\frac{2^{m+1}\left( l+m\right) !}{\left(
m-1\right) !(1-x^{2})^{\frac{m}{2}}}\sum_{n=0}^{\left[ \frac{l-m}{2}\right] }%
\frac{\left( -1\right) ^{n+m}\left( 2m+n-1\right) !P_{l+m}^{2m+2n}\left(
x\right) }{\left( 2n+l+3m\right) !n!} 
\]%
and an additional connection relation \newline
\[
P_{l}^{m}\left( x\right) =\frac{\left( -2\right) ^{-m}\left( l+m\right) !m!}{%
(1-x^{2})^{-\frac{m}{2}}}\sum_{n=\max \left( \left[ \frac{3m-l+1}{2}\right]
,0\right) }^{\min \left( \left[ \frac{l+m}{2}\right] ,2m\right) }\frac{%
P_{l-m}^{2m-2n}\left( x\right) }{\left( 2m-n\right) !\left( l+m-2n\right) !n!%
} 
\]%
are derived from (\ref{new12}) as $\mu =\pm m$ and $k=l\mp m$.

Formulas expressing associated Legendre functions as sums of products of
associated Legendre functions follow from (\ref{G0}): 
\begin{eqnarray*}
P_{k}^{m}\left( x\right) &=&\frac{r!\left( k-r\right) !\left( k+m\right) !}{%
\left( -1\right) ^{m}k!}\sum_{n=\max \left( 0,2r-m-k\right) }^{\min \left(
k-m,2r\right) }\frac{P_{r}^{n-r}\left( x\right) P_{k-r}^{r-m-n}\left(
x\right) }{\left( k-m-n\right) !n!}\text{,} \\
\text{ }r &\in &\mathbb{N}_{0}\text{, }0<r<k\text{;}
\end{eqnarray*}%
\begin{eqnarray*}
P_{k}^{m}\left( x\right) &=&\frac{r!\left( k-r\right) !\left( k+m\right) !}{%
k!}\sum_{n=\max \left( 0,2r+m-k\right) }^{\min \left( k+m,2r\right) }\frac{%
P_{r}^{n-r}\left( x\right) P_{k-r}^{r+m-n}\left( x\right) }{\left(
k+m-n\right) !n!}\text{,} \\
\text{ }r &\in &\mathbb{N}_{0}\text{, }0<r<k\text{.}
\end{eqnarray*}

The following expansions of associated Legendre functions into sums of
Gegenbauer polynomials are found from (\ref{P2}) and (\ref{P1}) as $\nu
=-k-1 $:%
\[
P_{k}^{m}\left( x\right) =\frac{\Gamma \left( k+\frac{3}{2}\right) \left(
m+k\right) !}{2^{m-1}\sqrt{\pi }\left( 1-x^{2}\right) ^{\frac{m}{2}}}%
\sum_{j=0}^{\left[ \frac{m+k}{2}\right] }\frac{\left( -1\right)
^{j}C_{k+m-2j}^{\left( k-m+1\right) }\left( x\right) }{\left( 2k-j+1\right)
!j!}\text{,} 
\]%
\[
\left( \frac{1+x}{1-x}\right) ^{\frac{m}{2}}P_{k}^{-m}\left( x\right) =\frac{%
\left( 2k+2m+1\right) !}{2^{2k}\left( k+1\right) _{m}}\sum_{n=0}^{k}\alpha
_{n}C_{n}^{\left( k+1\right) }\left( x\right) \text{,} 
\]%
where coefficients $\alpha _{n}$ are given by%
\[
\alpha _{n}=\frac{F\left( n-k,2m-2k-1,n+k+2m+2,-1\right) }{\left(
n+k+2m+1\right) !\left( k-n\right) !}\text{.} 
\]%
Another expression of associated Legendre functions in terms of Gegenbauer
polynomials,%
\[
P_{k+m}^{m}\left( x\right) =\frac{\left( 2m+1\right) !}{\left( k+2m+1\right)
m!}\sum_{n=0}^{\left[ \frac{k}{2}\right] }\frac{\left( 2n\right) !\left(
1-x^{2}\right) ^{\frac{2n+m}{2}}}{\left( -1\right) ^{n+m}2^{2n+m}n!}%
C_{k-2n}^{\left( n+m+1\right) }\left( x\right) \text{,} 
\]%
is obtained from (\ref{con!5}).

As $\nu =k+m$, $\mu =m$, and $l-n=j$, we get from (\ref{gen1aa}) as $s\in 
\mathbb{C}$, 
\[
\sum_{j=0}^{k}\frac{\left( -1\right) ^{k+j}s^{k-j}P_{j+m}^{m}\left( x\right) 
}{\left( j+2m\right) !\left( k-j\right) !}=\frac{\widehat{P}_{k+m}^{m}\left( 
\frac{x-s}{\sqrt{1-2sx+s^{2}}}\right) }{\left( k+2m\right) !\left(
1-2sx+s^{2}\right) ^{-\frac{k+m}{2}}}\text{,} 
\]%
and hence

\begin{eqnarray*}
P_{k}^{m}\left( x\right) &=&\left( k+m\right) !\sum_{j=0}^{k-m}\frac{\left(
-1\right) ^{j}\left( 2x\right) ^{k-m-j}P_{j+m}^{m}\left( x\right) }{\left(
j+2m\right) !\left( k-m-j\right) !}\text{,} \\
P_{k}^{m}\left( \sqrt{\frac{1\pm x}{2}}\right) &=&\frac{\left( k+m\right) !}{%
2^{\frac{k}{2}}\left( 1\pm x\right) ^{\frac{k}{2}}}\sum_{j=0}^{k-m}\frac{%
\left( -1\right) ^{\frac{j\mp j}{2}}P_{j+m}^{m}\left( x\right) }{\left(
j+2m\right) !\left( k-m-j\right) !}\text{,} \\
P_{k}^{m}\left( \sqrt{1-x^{2}}\right) &=&\frac{\left( k+m\right) !}{\left(
1-x^{2}\right) ^{\frac{k}{2}}}\sum_{j=0}^{k-m}\frac{\left( -1\right)
^{j}x^{j}P_{j+m}^{m}\left( x\right) }{\left( j+2m\right) !\left(
k-m-j\right) !}\text{,} \\
\frac{\widehat{P}_{k}^{m}\left( \cos \alpha \right) }{\left( k+2m\right) !}
&=&\sum_{j=0}^{k}\frac{\sin ^{k-j}\left( \beta -\alpha \right) }{\left(
j+2m\right) !\left( k-j\right) !}\frac{\sin ^{m+j}\alpha }{\sin ^{k+m}\beta }%
P_{j+m}^{m}\left( \cos \beta \right) \text{.}
\end{eqnarray*}

As $\nu =k$ and $\mu =\pm m$, (\ref{gen2a}) leads to 
\[
\sum_{n=0}^{k\pm m}\frac{P_{k}^{n\mp m}\left( x\right) }{\left(
1-x^{2}\right) ^{\frac{n\mp m}{2}}}\frac{s^{n}}{n!}=\frac{\widehat{P}%
_{k}^{\mp m}\left( x-s\right) }{\left( 1-\left( x-s\right) ^{2}\right) ^{\mp 
\frac{m}{2}}}\text{,} 
\]%
Now, (\ref{uuu1}),(\ref{uuu2}), (\ref{uuu3}), and (\ref{uuu4}) turn into
finite sums and give us another sort of relations for associated Legendre
functions. In particular, as $\theta +\varphi \in \left( 0,\pi \right) $,

\begin{eqnarray*}
\frac{P_{k}^{\mp m}\left( \cos \left( \theta +\varphi \right) \right) }{\sin
^{\mp m}\left( \theta +\varphi \right) } &=&\sum_{n=0}^{k\pm m}\frac{\sin
^{n}\theta \sin ^{n}\varphi }{n!}\frac{P_{k}^{n\mp m}\left( \cos \theta \cos
\varphi \right) }{\left( 1-\cos ^{2}\theta \cos ^{2}\varphi \right) ^{\frac{%
n\mp m}{2}}}\text{,} \\
\frac{2^{\pm m}P_{k}^{\mp m}\left( \cos \theta \cos \varphi \right) }{\left(
1-\cos ^{2}\theta \cos ^{2}\varphi \right) ^{\mp \frac{m}{2}}}
&=&\sum_{n=0}^{k\pm m}\frac{\left( -1\right) ^{n}\cos ^{n}\left( \theta
-\varphi \right) P_{k}^{n\mp m}\left( \frac{1}{2}\cos \left( \theta +\varphi
\right) \right) }{\left[ 4-\cos ^{2}\left( \theta +\varphi \right) \right] ^{%
\frac{n\mp m}{2}}n!}\text{.}
\end{eqnarray*}%
The generation function (\ref{gen3a}) takes the form 
\[
\sum_{n=0}^{k\pm m}P_{k}^{n\mp m}\left( x\right) \frac{s^{n}}{n!}=\frac{%
\widehat{P}_{k}^{\mp m}\left( x-s\sqrt{1-x^{2}}\right) }{\left( 1+\frac{2sx}{%
\sqrt{1-x^{2}}}-s^{2}\right) ^{\mp \frac{m}{2}}}\text{,} 
\]%
and then as $s=i$ and $x=\cos \theta $, $\theta \in \left( 0,\pi \right) $, 
\[
\widehat{P}_{k}^{\mp m}\left( e^{\pm i\theta }\right) =2^{\mp \frac{m}{2}}e^{%
\frac{i\pi m}{4}}e^{-\frac{im\theta }{2}}\sin ^{\pm \frac{m}{2}}\theta
\sum_{n=0}^{k\pm m}\frac{e^{\mp \frac{i\pi n}{2}}}{n!}P_{k}^{n\mp m}\left(
\cos \theta \right) \text{.} 
\]

The addition theorems (\ref{adi1}) and (\ref{adi2}) can be rewritten in the
forms of addition theorems for associated Legendre functions: as $0<x<1$ and 
$0<x+y<1$,

\begin{eqnarray*}
\frac{P_{k}^{m}\left( \frac{x+y}{1+xy}\right) }{\left( 1+xy\right) ^{-k}} &=&%
\frac{\left( k+m\right) !}{\left( k-1\right) !}\sum_{n=0}^{k\mp m}\frac{%
\left( k+n-1\right) !\Phi _{n,k,\pm m}\left( y\right) \left( \frac{1-x}{1+x}%
\right) ^{\frac{n}{2}}}{\left( -1\right) ^{n-\frac{m}{2}\pm \frac{m}{2}%
}\left( k\pm m+n\right) !n!}P_{k}^{n\pm m}\left( x\right) \text{,} \\
\Phi _{n,\pm m}\left( y\right) &=&\left( 1-y\right) ^{\pm \frac{m}{2}}\left(
1+y\right) ^{k\mp \frac{m}{2}}F\left( -n,-k;1-k-n;\frac{1-y}{1+y}\right) 
\text{;}
\end{eqnarray*}

\begin{eqnarray*}
\frac{P_{k}^{m}\left( \frac{x+y}{1+xy}\right) }{\left( 1+xy\right) ^{-k}} &=&%
\frac{\left( k+m\right) !k!}{\left( k-m\right) !\left( 2m-1\right) !}%
\sum_{n=0}^{k}\frac{\left( -1\right) ^{n}\Lambda _{n,k,m}\left( y\right)
P_{k+m}^{m+n}\left( x\right) }{n!\left( 1-x\right) ^{\frac{k-n}{2}}\left(
1+x\right) ^{\frac{n+k}{2}+m}}\text{,} \\
\Lambda _{n,k,m}\left( y\right) &=&\frac{\left( 2m+n-1\right) !}{\left(
k+2m+n\right) !}\left( \frac{1-y}{1+y}\right) ^{m-k}F\left( -n,m-k;2m;\frac{%
2y}{1+y}\right) \text{;}
\end{eqnarray*}%
\begin{eqnarray*}
\frac{P_{k}^{m}\left( \frac{x+y}{1+xy}\right) }{\left( 1+xy\right) ^{k+1}}
&=&\frac{\left( 2m\right) !}{k!}\sum_{n=0}^{2m}\frac{\left( k+n\right) !%
\widetilde{\Lambda }_{n,k,m}\left( y\right) \left( 1-x\right) ^{\frac{n+k+1}{%
2}}}{n!\left( 2m-n\right) !\left( 1+x\right) ^{\frac{n-k-1}{2}-m}}%
P_{k+m}^{m-n}\left( x\right) \text{,} \\
\widetilde{\Lambda }_{n,k,m}\left( y\right) &=&\left( \frac{1-y}{1+y}\right)
^{k-m+1}F\left( -n,k-m+1;-2m;\frac{2y}{1+y}\right) \text{;}
\end{eqnarray*}%
\begin{eqnarray*}
\frac{P_{k}^{m}\left( \frac{x+y}{1+xy}\right) }{\left( 1+xy\right) ^{-k}}
&=&k!\left( 2m\right) !\sum_{n=\max \left( 0,k-2m\right) }^{\min \left(
k,2m\right) }\frac{\left( -1\right) ^{n}\Theta _{n,k,m}\left( y\right)
\left( 1-x\right) ^{\frac{n-k}{2}}P_{k-m}^{m-n}\left( x\right) }{\left(
k-n\right) !\left( 2m-n\right) !n!\left( 1+x\right) ^{\frac{n+k}{2}-m}}\text{%
,} \\
\Theta _{n,k,m}\left( y\right) &=&\left( \frac{1+y}{1-y}\right)
^{m+k}F\left( -n,-k-m;-2m;\frac{2y}{1+y}\right) \text{.}
\end{eqnarray*}

By setting $\gamma =-k$, $\mu =\pm m$, and using (\ref{P-m0}), (\ref{Jac1})
and (\ref{Geg0}), the addition theorem (\ref{sh-theor}) turns into the
expansion

\begin{eqnarray*}
P_{k}^{\mp m}\left( \frac{x+y}{1+xy}\right) &=&\Psi \left( x,y\right)
\sum_{n=0}^{k}\frac{\Gamma \left( \frac{1}{2}-n\mp m\right) \mathfrak{A}%
_{n}^{\left( \pm m-k\right) }\left( y\right) }{2^{n}\left( k-n\right) !n!}%
C_{k\pm m}^{\left( \frac{1}{2}-n\mp m\right) }\left( x\right) \text{,} \\
\Psi \left( x,y\right) &=&\frac{\left( -1\right) ^{m}k!}{\sqrt{\pi }2^{\pm
m}\left( 1+xy\right) ^{k}}\left( \frac{1-y^{2}}{1-x^{2}}\right) ^{\pm \frac{m%
}{2}}\text{, }-1<x,y<1\text{,}
\end{eqnarray*}%
where the polynomials $\mathfrak{A}_{n}^{\left( \pm m-k\right) }(z)$ are
defined in (\ref{Ur}).

\end{document}